\documentclass[english, 11pt]{elsarticle}
\usepackage[T1]{fontenc}
\usepackage[latin9]{inputenc}
\usepackage{mathrsfs}
\usepackage{mathtools}
\usepackage{bm}
\usepackage{amsmath}
\usepackage{amsthm}
\usepackage{amssymb}

\makeatletter
\numberwithin{equation}{section}
\numberwithin{figure}{section}
\theoremstyle{plain}
\newtheorem{thm}{\protect\theoremname}[section]
\theoremstyle{plain}
\newtheorem{cor}[thm]{\protect\corollaryname}
\theoremstyle{remark}
\newtheorem{rem}[thm]{\protect\remarkname}
\theoremstyle{plain}
\newtheorem{lem}[thm]{\protect\lemmaname}
\theoremstyle{definition}
\newtheorem{defn}[thm]{\protect\definitionname}
\theoremstyle{plain}
\newtheorem{prop}[thm]{\protect\propositionname}
\theoremstyle{remark}
\newtheorem*{rem*}{\protect\remarkname}

\usepackage[a4paper]{geometry}

\usepackage{times}

\usepackage{color}
\definecolor{red}{rgb}{1,0,0} 
\definecolor{green}{rgb}{0,1,0} 
\definecolor{blue}{rgb}{0,0,1} 
\definecolor{darkblue}{rgb}{0,0,0.6}
\definecolor{darkred}{rgb}{0.6,0,0}

\usepackage[colorlinks,
            linkcolor=blue,
            anchorcolor=darkblue,
            citecolor=darkblue
           ]{hyperref}

\newcommand{\eps}{\varepsilon}  
  
\newcommand{\vf}{\varphi}

\newcommand{\md}{\mathrm{d}}   
\newcommand{\vd}{\,\md}
\newcommand{\me}{\mathrm{e}}   

\newcommand{\R}{\mathbb{R}}   

\newcommand{\PS}{\Omega}  
\newcommand{\E}{\mathbb{E}}  
\newcommand{\Prob}{\mathbb{P}}  
\newcommand{\Filt}{\mathcal{F}}  

\newcommand{\ext}{\tau}  
\newcommand{\hit}{\gamma}  
\newcommand{\bd}{\lambda}  

\DeclareMathOperator*{\osc}{osc}

\makeatother

\usepackage{babel}
\providecommand{\corollaryname}{Corollary}
\providecommand{\definitionname}{Definition}
\providecommand{\lemmaname}{Lemma}
\providecommand{\propositionname}{Proposition}
\providecommand{\remarkname}{Remark}
\providecommand{\theoremname}{Theorem}

\begin{document}

\begin{frontmatter}{}

\title{Krylov--Safonov estimates for a degenerate diffusion process\footnote{F. Zhang was partially supported by National Natural Science Foundation
of China (Grants No. 11701369). K. Du was partially supported by National
Science Foundation of China (Grant No. 11801084).}}

\author{Fu Zhang}

\address{College of Science, University of Shanghai for Science and Technology,
334 Jungong Road, Shanghai 200093, China.}

\ead{fuzhang82@gmail.com}

\author{Kai Du}

\address{Shanghai Center for Mathematical Sciences, Fudan University, 2005
Songhu Road, Shanghai 200438, China.}

\ead{kdu@fudan.edu.cn}
\begin{abstract}
This paper proves a Krylov\textendash Safonov estimate for a multidimensional
diffusion process whose diffusion coefficients are degenerate on the
boundary. As applications the existence and uniqueness of invariant
probability measures for the process and H\"older estimates for the
associated partial differential equation are obtained.
\end{abstract}

\end{frontmatter}{}

\section{Introduction}

Assume that $X=\{(X_{t},\Prob^{x}):t\ge0,\,x=(x^{1},\dots,x^{n})\in\R_{+}^{n}\coloneqq[0,\infty)^{n}\}$
is a time-homogeneous strong Markov process on a measurable space
$(\PS,\Filt)$ with a filtration $\{\mathcal{F}_{t}\}_{t\ge0}$, whose
infinitesimal generator $\mathcal{L}$ is given by
\begin{equation}
\mathcal{L}f(x)=\frac{1}{2}\sum_{i,j=1}^{n}a^{ij}(x)\sqrt{x^{i}x^{j}}\frac{\partial^{2}f}{\partial x^{i}\partial x^{j}}(x)+\sum_{i=1}^{n}b^{i}(x)\frac{\partial f}{\partial x^{i}}(x)\quad\forall f\in C_{b}^{2}(\R_{+}^{n}),\label{eq:operator}
\end{equation}
where $a^{ij}=a^{ji}:\R_{+}^{n}\to\R^{n}$ and $b^{i}:\R_{+}^{n}\to\R$
are measurable and \emph{locally bounded }functions, and $C_{b}^{2}(\R_{+}^{n})$
denotes the space of bounded and twice differentiable functions defined
on $\R_{+}^{n}$. This process relates to a stochastic differential
equation (SDE) of the following form
\begin{equation}
\md X_{t}^{i}=b^{i}(X_{t})\vd t+\sqrt{X_{t}^{i}}\,\sum\nolimits_{k}\sigma^{ik}(X_{t})\vd W_{t}^{k},\quad i=1,\dots,n,\label{eq:SDE}
\end{equation}
where $W$ is a multidimensional Brownian motion and $\sum_{k}\sigma^{ik}(x)\sigma^{jk}(x)=a^{ij}(x)$.
It is worth noting that the diffusion coefficients of $X$ are degenerate
on the boundary $\partial\R_{+}^{n}$.

This paper aims to study the regularity of a class of functions characterized
by the Markov process $X$. It is well-known that a classical harmonic
function can be characterized via multi-dimensional Brownian motion
(see \citep{karatzas1991brownian} for example). Motivated by this
fact the concept of general harmonic functions associated with Markov
processes was proposed by Dynkin \citep{dynkin1981harmonic}; those
functions and further extensions often relate to elliptic and parabolic
partial differential equations (PDEs). In a word, there is a rich
interplay between probability theory and analysis; in this context,
the probabilistic method has been used to many problems from analysis
and PDEs with fruitful outcomes. A celebrated example is the Krylov\textendash Safonov
estimate for nondegenerate diffusion processes (cf. \citep{krylov1979estimate}),
yielding a fundamental estimate for the regularity theory of fully
nonlinear elliptic and parabolic equations. Adapting Krylov\textendash Safonov's
probabilistic approach, this paper shall prove the following regularity
result for functions associated with the degenerate diffusion process
$X$ in some way. In what follows, $B_{b}(E)$ denotes the set of
bounded Borel functions defined on a set $E$.
\begin{thm}
\label{thm:main1} Let $D\subset\R_{+}^{n}$ be a simply connected
open domain containing $\partial D\cap\partial\R_{+}^{n}$, and let
$Q=[0,1)\times D$ and $\tau_{Q}=\inf\{t>0:(t,X_{t})\notin Q\}$.
Assume that
\begin{itemize}
\item[\emph{(}\textbf{\emph{C}}\emph{)}]  for each $x\in D\cap\partial\mathbb{R}_{+}^{n}$ with $x^{i}=0$,
the function $b^{i}$ has a positive lower bound in a neighborhood
of $x$; and for each $x\in D$, the matrix-valued function $a=(a^{ij})$
is uniformly positive definite in a neighborhood of $x$. 
\end{itemize}
Then, as long as $u\in B_{b}(Q)$ satisfies that
\begin{itemize}
\item[\emph{(}\textbf{\emph{U}}\emph{)}]  there is an $f\in B_{b}(Q)$ such that for each $(t,x)\in Q$, the
process
\[
u(t+s\wedge\tau_{Q_{0}},X_{s\wedge\tau_{Q_{0}}})+\int_{0}^{s\wedge\tau_{Q_{0}}}\!\!f(t+r,X_{r})\vd r\ \ \text{with}\ s\ge0
\]
is a $\Prob^{x}$-martingale with respect to $\Filt_{s}$ ,
\end{itemize}
the function $u$ is locally H\"older continuous in $Q$; more specifically,
for any compact set $S\subset Q$ there exist constants $\alpha\in(0,1)$
and $C>0$, depending only on the set $S$ and the functions $a$
and $b=(b^{i})$, such that
\begin{equation}
|u(t,x)-u(s,y)|\le C\big(\|u\|_{L^{\infty}}+\|f\|_{L^{\infty}}\big)\big(|t-s|^{\alpha/2}+\max_{i}|\sqrt{x^{i}}-\sqrt{y^{i}}|^{\alpha}\big)\label{eq:holder}
\end{equation}
for all $(t,x)$ and $(s,y)$ in $S$.
\end{thm}

Condition (\textbf{U}) gives a characterization of certain functions
in terms of $X$; when $f=0$ and $u$ depends only on $x$, it is
equivalent to the definition of $X$-harmonic functions in the literature
(see \citep{dynkin1981harmonic,athreya2002degenerate} for example).
In a relevant work Athreya et al. \citep{athreya2002degenerate} proved
the pointwise continuity of bounded $X$-harmonic functions (see Theorem
6.4 there). The precise dependence of the dominating constant $C$
will be specified in the next section where the theorem is proved
with the help of an estimate of hitting times for $X$ (see Theorem
\ref{thm:main2} below).

This paper presents two direct applications of Theorem \ref{thm:main1},
which also partly motivated this work. The first one is the following
\emph{a priori} H\"older estimate for a linear PDE. Indeed, for a
function $u$ in the space $C^{1,2}(\bar{Q})$ of all functions on
$\bar{Q}$ having continuous time derivatives and second-order spatial
derivatives, one can apply It\^o's formula to $u(t,X_{t})$ to verify
Condition (\textbf{U}) with $f=\mathcal{L}u$, where the operator
$\mathcal{L}$ is given by (\ref{eq:operator}).
\begin{cor}
\label{cor:PDE}Under the assumptions of Theorem \ref{thm:main1},
if $u\in C^{1,2}(\bar{Q})$ and $f:=\mathcal{L}u\in B_{b}(Q)$, then
$u$ enjoys the estimate (\ref{eq:holder}).
\end{cor}

The significant of this result, like the original Krylov\textendash Safonov
estimate \citep{krylov1979estimate} (or see \citep{bass1998diffusions}
for a detailed description), is that the estimate of $u$'s H\"older
continuity norm does not depend on the smoothness of the coefficients
$a$ and $b$. This is the key point for the applications of such
estimates to fully nonlinear PDEs. Although analytic approaches to
the Krylov\textendash Safonov estimate (see \citep{Krylov1981A,Trudinger1980local})
were found soon after \citep{krylov1979estimate}, the techniques
developed from its original probabilistic proof are still powerful
to study nonlinear operators and nonlocal operators, see \citep{bass2002harnack,delarue2010krylov,chen2012boundary}
for example. Moreover, there are some relevant results in the literature
of PDEs, for instance, the Harnack inequalities and H\"older estimates
were proved in \citep{Daskalopoulos1998regularity,DaskalopoulosLee2003Holder,HongHuang2012Lp,Lieberman2016schauder}
for the equations that degenerate along one direction; those equations
stemmed from physics and geometry.

Another direct application of Theorem \ref{thm:main1} is to obtain
the existence and uniqueness of invariant probability measures for
$X$. For readers' convenience, let us recall some related notions
(cf. \citep{da1996ergodicity}). The transition semigroup $P=(P_{t})_{t\ge0}$
associated with the process $X$ is defined as
\[
P_{t}f(x)=\E^{x}f(X_{t}),\quad\forall\,f\in B_{b}(\R_{+}^{n});
\]
and a probability measure $\mu$ on $\R_{+}^{n}$ is called to be
\emph{invariant} with respect to $P$ if
\[
\mu(f)=P_{t}^{*}\mu(f)\coloneqq\int_{\R_{+}^{n}}P_{t}f(x)\,\mu(\md x),\quad\forall\,t>0,\,f\in B_{b}(\R_{+}^{n}).
\]
The invariant probability measure is an important concept in ergodic
theory of Markov processes, its existence and uniqueness can usually
be proved by means of the Krylov-Bogoliubov existence theorem and
the Doob\textendash Khas'minskii theorem (cf. \citep[Sections 4.1 and 4.2]{da1996ergodicity}),
and a key point is to show that the semigroup $P$ is \emph{strongly
Feller}, namely, $P_{t}f\in C(\R_{+}^{n})$ for some $t>0$ and $f\in B_{b}(\R_{+}^{n})$. 
\begin{thm}
\label{thm:inv-meas}Under Condition \emph{(}\textbf{\emph{C}}\emph{)}
the transition semigroup $P$ for the process $X$ is strongly Feller.
Moreover, if additionally there is a constant $\bd\ge1$ such that
for all $x\in\R_{+}^{n}$,
\begin{equation}
\bd I\geq a(x)\geq\bd^{-1}I,\quad\bd\geq b^{i}(x)\geq-\bd x^{i},\;i=1,2,\ldots,n,\label{eq:inv-meas-1}
\end{equation}
then $P$ has a unique invariant probability measure.
\end{thm}

Important applications of the degenerate diffusion process $X$ can
be found in the theory of superprocesses and in financial modeling.
It has been used to characterize a class of measure-valued diffusions
called super-Markov chains, which is the limit of a large branching
particle system with finite states (see \citep{athreya2002degenerate,bass2003degenerate}
for more details about super-Markov chains). In mathematical finance,
some special forms of $X$ and other similar processes were used to
model term structures of defaultable bonds, see \citep{duffie1999modeling,dai2000specification}
for details.

It is worth noting that existence of the process $X$ is not an outcome
but the major assumption in this work. This assumption is reasonable.
Actually, the construction of such a process can be converted to solving
a martingale problem of Stroock and Varadhan associated with the operator
$\mathcal{L}$ (cf. \citep{stroock1979multidimensional}); and for
the latter problem the proof of Theorem 1.1 in \citep[Section~7]{athreya2002degenerate}
(see also \citep[Remark 1.1(a)]{bass2003degenerate}) gives a standard
argument to show existence of solutions under that the coefficients
$a^{ij}$ and $b^{i}$ are continuous and satisfy Condition (\textbf{C}),
providing us with a strong support to our assumption, though we believe
that the smoothness requirement on the coefficients might be released
more or less.

Uniqueness of solutions to the martingale problem for $\mathcal{L}$,
though unnecessary in this paper, is very important both in theory
and in practice, but having not been solved completely under the same
condition for existence. It is simply valid when the coefficients
$a^{ij}$ and $b^{i}$ are constant due to the Yamada\textendash Watanabe
uniqueness theorem (cf. \citep{yamada1971uniqueness}), but seems
to be difficult when the coefficients are variable. Remarkable works
have been done in \citep{athreya2002degenerate,bass2003degenerate}
where the uniqueness was proved if $a^{ij}$ and $b^{i}$ are continuous
and the matrix $a=(a^{ij})$ is almost diagonal; they also gave a
comprehensive explanation how to reduce the uniqueness problem to
some sharp estimates for $\mathcal{L}$ with constant coefficients
by using Stroock and Varadhan's perturbation argument. Following this
strategy our working paper \citep{zhang2019} attempts to prove a
Schauder estimate for $\mathcal{L}$, effective for the concerned
uniqueness problem, based on the estimate (\ref{eq:holder}). This
is another motivation of this work.

To capture the essential difficulties caused by degeneracy, let us
briefly review Krylov and Safonov's original work \citep{krylov1979estimate}
for nondegenerate operators. A key observation is that the generator
of a diffusion process enjoys certain smoothing property if the paths
of the process sufficiently visit the surrounding space with a non-trivial
probability (see \citep[Page 926]{delarue2010krylov} for an intuitive
explanation). To be more specific, we consider, for simplicity, a
strong Markov process $Y=(Y_{t},\mathbb{Q}^{y})$ with generator $\mathcal{A}=\sum_{i,j=1}^{n}\tilde{a}^{ij}(y)\partial_{ij}$,
where $\tilde{a}=(\tilde{a}^{ij})$ is bounded and uniformly positive
definite. Let $K_{r}(y)=\{z:|z^{i}-y^{i}|<r,\,i=1,\dots,n\}$ and
$\Gamma\subset\R^{n}$ a Borel set, and define the exit time $\tau_{r}=\inf\{t>0:Y_{t}\notin K_{r}(y)\}$
and the hitting time $\gamma_{\Gamma}=\inf\{t>0:Y_{t}\in\Gamma\}$.
If one can obtain a lower bound of the hitting probability of $\Gamma$
within $K_{r}(y)$, namely, $\mathbb{Q}^{y}[\gamma_{\Gamma}<\tau_{r}]>\eps>0$
for all $\Gamma$ with $|\Gamma\cap K_{r}(y)|>\mu|K_{r}(y)|$ and
$\mu>0$, then a $Y$-harmonic function is H\"older continuous at
the point $y$. Furthermore, if the constant $\eps$ depends only
on $\mu$ and the upper and lower bounds of $\tilde{a}$ but not on
$y$ and $r$, then the H\"older continuity is uniform: it is simply
valid in this case because by translation and rescaling it suffices
to prove the estimate only for $y=0$ and $r=1$. Readers are referred
to \citep[Section V.7]{bass1998diffusions} for detailed arguments.
We remark that the uniform estimate of hitting probability heavily
relies on the uniform boundedness and positive definiteness of $\tilde{a}$
in the nondegenerate case.

So there were two major issues to be tackled in our problem: estimating
the hitting probability when the process starts from boundary where
$\mathcal{L}$ is degenerate, and uniformity of the estimate. The
issues are intertwined in some sense. Indeed, the first one was addressed
in \citep[Theorem 6.4]{athreya2002degenerate}, without considering
uniformity, to prove the pointwise continuity of $X$-harmonic functions.
Their approach made a careful use of Krylov and Safonov's estimate,
based on an important property of $X$ that the process would be pulled
inside rapidly by the drift term (recalling that $b^{i}>0$ near $\{x^{i}=0\}$)
if it is at or runs towards the boundary, but their estimate was not
uniform because of its dependency on the starting point and the size
of the neighborhood. Such a ``pulling-back'' property also plays
a key role in our estimating of hitting probability. In order to obtain
a uniform estimate, we proceed Krylov and Safonov's original argument
with some substantial changes. In terms of rescaling we have two observations.
First, for all $r>0$, the rescaled process $(r^{-1}X_{rt})_{t\geq0}$
has the same structure required in Condition (\textbf{C}); in other
words, the estimates for both hitting probability and H\"older continuity
must be invariant under rescaling $(t,x)\mapsto(rt,rx)$. Second,
in an area keeping a positive distance from the boundary $\partial\R_{+}^{n}$,
the process $\sqrt{X}=(\sqrt{X^{1}},\dots,\sqrt{X^{n}})$ satisfies
the condition of Krylov and Safonov's original result, which implies,
if $u$ satisfies Condition (\textbf{U}), then in this area the function
$v(t,x)=u(t,x^{2})$ must be $\alpha$-H\"older in $x$ and $\frac{\alpha}{2}$-H\"older
in $t$. According to these observation, the form of estimate (\ref{eq:holder})
is appropriate for our problem; correspondingly, we introduce in our
proof a class of anisotropic hypercubes instead of the hypercubes
$K_{r}(y)$ in the nondegenerate case, which matches the above scaling
properties (see (\ref{eq:cube}) and Remark \ref{rem:secaling} below
for details). As a result, these changes make the argument more delicate
and involved than that for nondegenerate diffusion processes; for
example, we must estimate hitting probability for any starting point,
and carefully determine the dominating constants so that they do not
depend on the starting point.

This paper is organized as follows: Section 2 proves Theorem \ref{thm:main1}
based on an estimate of hitting time for the process $X$ (Theorem
\ref{thm:main2} below); Section 3 gives several auxiliary results,
including some estimates for $X$ and a measure theory lemma; Section
4 estimates the hitting time for large target sets; Section 5 completes
the proof of Theorem \ref{thm:main2}; and Section 6 proves Theorem
\ref{thm:inv-meas}.

We finish this section with some comments on the setting of this work
and notation used in what follows. Notice that the Markov process
$X=(X_{t},\mathbb{P}^{x})$ can induce a family of probability measures
on the canonical space $C([0,\infty),\R_{+}^{n})$, still denoted
by $\mathbb{P}^{x}$, under which the coordinate process is identical
to $X$ in law. Since our main result only depends on the law of $X$,
we can simply take $\Omega=C([0,\infty),\R_{+}^{n})$ and $X_{t}(\omega)=\omega(t)$,
and for $t\ge0$ and $x\in\mathbb{R}_{+}^{n}$, define the probability
measure $\Prob^{t,x}$ on $\Omega$ such that $\mathbb{P}^{t,x}[X_{t+s}\in A]=\mathbb{P}^{x}[X_{s}\in A]$
for all $s\ge0$ and Borel set $A\subset\R_{+}^{n}$; then for any
$f\in B_{b}(\R_{+}^{n})$ we have $\E^{t,x}f(X_{t+s})=\E^{x}f(X_{s})$.

\section{Proof of Theorem \ref{thm:main1}}

The proof of Theorem \ref{thm:main1} is based on a result (Theorem
\ref{thm:main2} below) concerning the probability that $X$ hits
a set of positive measure. Let us introduce some notation: for 
\begin{align*}
\theta & \in(0,1],\quad\rho>0,\quad(t,x)\in[0,\infty)\times\mathbb{R}_{+}^{n},
\end{align*}
we denote 
\[
L^{i}(x^{i},\rho):=\begin{cases}
\big[0,[\sqrt{x^{i}}+\rho]^{2}\big), & {\rm if}\,\sqrt{x^{i}}\leq\rho;\\
\big([\sqrt{x^{i}}-\rho]^{2},[\sqrt{x^{i}}+\rho]^{2}\big), & {\rm if}\,\sqrt{x^{i}}>\rho,
\end{cases}
\]
and define the anisotropic cubes 
\[
K(x,\rho):=\prod_{i=1}^{n}L^{i}(x^{i},\rho),
\]
and the anisotropic hypercubes
\begin{equation}
Q_{\theta}(t,x,\rho):=[t,t+\theta\rho^{2})\times K(x,\rho).\label{eq:cube}
\end{equation}
We call the number $\rho$ to be the \emph{size} of $K(x,\rho)$ and
$Q_{\theta}(t,x,\rho)$.
\begin{rem}
\label{rem:secaling} (1) The set $Q_{\theta}(t_{0},x,\rho)$ are
consistent under the rescaling 
\[
(t_{0}+t,x)\mapsto(t_{0}+rt,rx)
\]
 with $r>0$, for instance,
\begin{equation}
(t_{0},0)+rQ_{\theta}(0,x,\rho)=Q_{\theta}(t_{0},rx,\sqrt{r}\rho).\label{eq:rescaling}
\end{equation}

(2) Suppose $(X_{t})_{t\geq t_{0}}$ is a process satisfies SDE (\ref{eq:SDE}).
Obviously process $\big(\tilde{X}_{t}=\rho^{-2}X_{t_{0}+\rho^{2}t}\big)_{t\geq0}$.
satisfies 
\[
\vd\tilde{X}_{t}^{i}=\tilde{b^{i}}(\tilde{X}_{t})\md t+\sqrt{\tilde{X}_{t}^{i}}\langle\tilde{\sigma^{i}}(\tilde{X}_{t})\cdot\md\tilde{W}_{t}\rangle,
\]
where the rescaled process $\tilde{W}_{t}=\rho^{-1}W_{t_{0}+\rho^{2}t}$
is also a standard Brownian motion, and $(\tilde{b}(\cdot),\tilde{\sigma}(\cdot)):=(b(\rho^{2}\cdot),\sigma(\rho^{2}\cdot))$
has the same law on $Q_{\theta}(t_{0},x,\rho)$ as $(b,\sigma)$ on
$Q_{\theta}(t_{0},x,1)$. It means that $\tilde{X}$ and $X$ share
the same properties respectively on $Q_{\theta}(t_{0},x,\rho)$ and
$Q_{\theta}(t_{0},x,1)$.

(3) The length of edges of hypercubes $Q_{\theta}(t,x,\rho)$ depends
not only on the size $\rho$ but also on $x$. The length of $Q_{\theta}(t,x,\rho)$
along the $i$-th coordinate direction is increasing with respect
to $x^{i}$.
\end{rem}

We define the \emph{hitting time} for a Borel set $\Gamma$ on event
$\{X_{t}=x\}$
\[
\hit_{\Gamma}=\hit_{\Gamma}^{t,x}=\inf\{s>t:X_{s}\in\Gamma,X_{t}=x\}
\]
and the\emph{ exiting time} for a hypercube $Q$
\[
\ext_{Q}=\ext_{Q}^{t,x}=\inf\{s>t:X_{s}\notin Q,X_{t}=x\}.
\]
It is known that $\hit_{\Gamma}$ and $\ext_{Q}$ are both stopping
times (c.f. \citep[Theorem 2.4]{bass2010measurability}) under condition
$\{X_{t}=x\}$.

We may use a more precise form of Condition \textbf{(C)} as follows:
\begin{itemize}
\item[(\textbf{C'})]  Given $x_{0}\in\R_{+}^{n}$ and $\rho\in(0,1)$ there is a constant
$\bd>1$ such that
\[
\bd^{-1}I_{n}\leq a\leq\bd I_{n},\quad|b|\leq\bd\quad\text{on }\,K(x_{0},\rho)
\]
 and 
\[
b^{i}\geq\bd^{-1},\quad{\rm if}\,\sqrt{x^{i}}\in[0,\rho]\cap[(\sqrt{x_{0}^{i}}-\rho),(\sqrt{x_{0}^{i}}+\rho)].
\]
\end{itemize}
\begin{thm}
\label{thm:main2} Let Condition \emph{(}\textbf{\emph{C'}}\emph{)}
be satisfied. Then for any $\theta\in(0,1]$ and $\mu\in(0,1)$, there
exists a constant $\varepsilon=\eps(n,\bd,\theta,\mu)\in(0,1)$ such
that for any $x\in K(x_{0},\rho/6)$ and any closed set $\Gamma\subset Q:=Q_{\theta}(0,x_{0},\rho)$
satisfying $|\Gamma|\geq\mu|Q|$,
\[
\Prob^{x}[\hit_{\Gamma}\leq\ext_{Q}]\geq\varepsilon,
\]
where $x_{0}\in\mathbb{R}_{+}^{n}$ and $\rho\in(0,1]$ are arbitrarily
given.
\end{thm}

Sections 3\textendash 5 are devoted to the proof of the above theorem.
With its help one can prove Theorem \ref{thm:main1}.
\begin{proof}[Proof of Theorem \ref{thm:main1}]
 If Condition (\textbf{C}) holds, then there is $\rho_{0}\in(0,1)$
such that for any $(t_{0},x_{0})\in S$, the hypercube $Q_{1}(t_{0},x_{0},\rho_{0})\subset Q$
and satisfies Condition (\textbf{C'}) for some $\bd$ (obviously,
$\bd$ may depend on $S$).

For any $Q_{1}(t_{0},x_{0},\rho_{0})$, it suffices to prove that
for any $\rho\in(0,\rho_{0}]$,
\begin{equation}
\osc_{Q_{1}(t_{0},x_{0},\rho/6)}(u)\leq\nu\osc_{Q_{1}(t_{0},x_{0},\rho)}(u)+\rho^{2}\|f\|_{\infty}\label{eq:2-1}
\end{equation}
with some constant $\nu\in(0,1)$ independent of $\rho$ and $(t_{0},x_{0})\in S$.
Indeed, according to \citep[Lemma 4.6]{lieberman1996second}, it follows
from (\ref{eq:2-1}) that
\begin{equation}
\osc_{Q_{1}(t_{0},x_{0},\rho)}(u)\leq C\rho^{\delta}\Big(\osc_{Q_{1}(t_{0},x_{0},\rho_{0})}(u)+\|f\|_{\infty}\Big)\label{eq:2-2}
\end{equation}
for some $\delta\in(0,1)$ and any $\rho\in(0,\rho_{0}/6)$, and the
estimate (\ref{eq:holder}) follows immediately.

To prove (\ref{eq:2-1}), we set 
\[
m_{-}:=\inf_{Q_{1}(t_{0},x_{0},\rho)}(u)\quad\text{and}\quad m_{+}:=\sup_{Q_{1}(t_{0},x_{0},\rho)}(u)
\]
We may assume that 
\[
|\{(t,x)\in Q_{1}(t_{0},x_{0},\rho):u(t,x)\leq(m_{-}+m_{+})/2\}|\geq(1/2)|Q_{1}(t_{0},x_{0},\rho)|
\]
otherwise we consider $-u$ instead. For $t\in[t_{0},t_{0}+\rho^{2}/36]$,
set 
\begin{align*}
Q_{0} & :=Q_{35/36}(t,x_{0},\rho),\\
\Gamma & :=\{(s,y)\in Q_{0}:u(s,y)\leq(m_{-}+m_{+})/2\}.
\end{align*}
It is easily seen that
\[
|\Gamma|\geq(17/35)|Q_{0}|.
\]
Let $\hit_{\Gamma}$ and $\ext_{Q_{0}}$ be the associated hitting
and exiting times of $X$ starting from $(t,x)\in Q_{1}(t_{0},x_{0},\rho/6)$. 

With $\ext:=\hit_{\Gamma}\wedge\ext_{Q_{0}}$ , it follows from Condition
(\textbf{U}) and the optional stopping theorem that
\begin{equation}
u(t,x)=\E^{t,x}u(\ext,X_{\ext})+\E^{t,x}\int_{t}^{\tau}\!\!f(r,X_{r})\vd r.\label{eq:key}
\end{equation}
Then applying Theorem \ref{thm:main2} with $\theta=35/36$ and $\mu=17/35$,
we have 
\begin{eqnarray*}
u(t,x) & \leq & \E^{t,x}\big[u(\ext,X_{\ext})(\mathbf{1}_{\{\hit_{\Gamma}<\ext_{Q}\}}+\mathbf{1}_{\{\hit_{\Gamma}\geq\ext_{Q}\}})\big]+\rho^{2}\|f\|_{\infty}\\
 & \leq & \varepsilon\cdot\frac{m_{-}+m_{+}}{2}+(1-\varepsilon)m_{+}+\rho^{2}\|f\|_{\infty},
\end{eqnarray*}
 thus, 
\[
u(t,x)-m_{-}\leq(1-\eps/2)(m_{+}-m_{-})+\rho^{2}\|f\|_{\infty}.
\]
Therefore, (\ref{eq:2-1}) holds with $\nu=1-\varepsilon/2$ for every
$(t,x)\in Q_{\theta}(t_{0},x_{0},\rho/6)$.
\end{proof}

\section{Auxiliary results}

In what follows we may assume that the process $X$ satisfies SDE
(\ref{eq:SDE}). Indeed, our argument only depends on the law of $X$,
so we can select other proper copies of $X$ if necessary; on the
other hand, the process $X$, of which we have assumed the existence,
can induce a solution to the martingale problem for $\mathcal{L}$,
and, owing to a celebrated result of Stroock and Varadhan (see \citep[Corollary 5.4.8]{karatzas1991brownian}
for example), a weak solution of SDE (\ref{eq:SDE}), both identical
in law to $X$. 

\subsection{Some estimates for the process $X$}

We first derive some estimates for $1$-dimensional general squared
Bessel process. 
\begin{lem}
\label{lem:pre-est}Let $\alpha$ and $\beta$ be predictable processes
with
\begin{equation}
\lambda^{-1}\le|\alpha_{t}|^{2}\leq\lambda\quad\text{and}\quad|\beta_{t}|\le\lambda,\quad t\ge0\label{eq:bsigma}
\end{equation}
for some constant $\lambda\ge1$, and let $B$ be a Brownian motion
under a probability $\Prob$, and the process $Z$ satisfy
\[
\md Z_{t}=\beta_{t}\vd t+\alpha_{t}\sqrt{Z_{t}}\vd B_{t},\quad Z_{0}=z\ge0.
\]
Let $\eps\in(0,1)$ and $c>0$ be constants. Then we have the following
assertions:

\emph{(a)} There exists a constant $\kappa=\kappa(\eps,c,\lambda)\in(0,1)$
such that
\[
\Prob\Big[\sup_{0\le t\le\kappa\rho^{2}}|\sqrt{Z_{t}}-\sqrt{z}|\ge c\rho\Big]\le\eps
\]
for all $z\ge0$ and $\rho\in(0,1]$.

\emph{(b)} Suppose $\beta_{t}\geq\lambda^{-1}$ for all $t\ge0$ additionally,
and let $S$ be a random variable uniformly distributed on $[\bar{t},2\bar{t}]$
with $\bar{t}>0$ and independent of $\alpha$, $\beta$ and $B$.
Then there exists a constant $\xi=\xi(\bar{t},\lambda,\eps)>0$ such
that 
\[
\Prob[Z_{S}\le\xi]\le\eps.
\]
\end{lem}

\begin{proof}
Assertion (b) is taken from Lemma 6.2 in \citep{athreya2002degenerate}.
To prove (a), we consider $\rho=1$ first. Define $\tau:=\inf\{t:|\sqrt{Z_{t}}-\sqrt{z}|\ge1\}$.
By Chebyshev's inequality we have
\[
\Prob\Big[\sup_{0\le t\le s}|\sqrt{Z_{t}}-\sqrt{z}|\ge c\Big]\le\frac{1}{c}\E\Big[\sup_{0\le t\le s}|\sqrt{Z_{t\wedge\tau}}-\sqrt{z}|^{2}\Big].
\]

For $z\ge2$, using the equation of $Y_{t}=\sqrt{Z_{t\wedge\tau}}$:
\[
\md Y_{t}=\bm{1}_{\{t\le\tau\}}\frac{4\beta_{t}-|\alpha_{t}|^{2}}{8Y_{t}}\vd t+\bm{1}_{\{t\le\tau\}}\frac{\alpha_{t}}{2}\vd B_{t},
\]
one can easily obtain that
\[
\E\Big[\sup_{0\le t\le s}|Y_{t}-\sqrt{z}|^{2}\Big]\le C(\lambda)s.
\]

For $z<2$, by the relation $|\sqrt{a}-\sqrt{b}|^{2}\le|a-b|$ and
the Burkholder\textendash Davis\textendash Gundy (BDG) inequality,
one has
\begin{align*}
\E\Big[\sup_{0\le t\le s}|\sqrt{Z_{t\wedge\tau}}-\sqrt{z}|^{2}\Big] & \le\E\Big[\sup_{0\le t\le s}|Z_{t\wedge\tau}-z|\Big]\\
 & \le\lambda s+\E\sup_{0\le t\le s}\int_{0}^{t}\bm{1}_{\{r\le\tau\}}\alpha_{t}\sqrt{Z_{r}}\vd B_{t}\\
 & \le\lambda s+C(\lambda)\,\E\biggl(\int_{0}^{s}Z_{r\wedge\tau}\vd t\biggr)^{\!1/2}\\
 & \le\lambda s+C(\lambda)\sqrt{s}.
\end{align*}

To sum up one obtains that
\[
\Prob\Big[\sup_{0\le t\le s}|\sqrt{Z_{t}}-\sqrt{z}|\ge c\Big]\le\frac{C(\lambda)(s+\sqrt{s})}{c},
\]
so there is a constant $s=\kappa=\kappa(\eps,c,\lambda)\in(0,1)$
such that $C(\lambda)(s+\sqrt{s})/c\le\eps$, and we conclude the
case $\rho=1$. The case of general $\rho\in(0,1]$ can be obtained
by rescaling $\tilde{Z}_{t}=\rho^{-2}Z_{\rho^{2}t}$. 
\end{proof}
Let us turn to the estimates for the strong Markov process $X$. 
\begin{lem}
\label{lem:supp} Let $\beta>1$, $0<c\leq1$, $\alpha>\epsilon>0$.
Let Condition \textbf{\emph{(}}\emph{C'}\textbf{\emph{)}} be satisfied.
Then, for any $x,z\in\mathbb{R}_{+}^{n}$ and $l\in(0,1]$ with $0<cl\le\min_{\,i}\big\{\sqrt{x^{i}},\sqrt{z^{i}}\big\}$
and $\max_{\,i}|\sqrt{x^{i}}-\sqrt{z^{i}}|\le\beta l$, there is a
constant $m_{1}=m_{1}(c,\epsilon,\alpha,\beta,\bd)>0$ such that
\begin{equation}
\Prob^{z}\Bigg[\begin{gathered}\sup_{\epsilon l^{2}\le s\le\alpha l^{2}}\max_{i}|\sqrt{X_{s}^{i}}-\sqrt{x^{i}}|\le3cl/4,\\
X_{s}\in\tilde{K}(x,z;3cl/4)\:\forall s\in[0,\alpha l^{2}]
\end{gathered}
\Bigg]\ge m_{1}(c,\epsilon,\alpha,\beta,\bd).\label{eq:supp-1}
\end{equation}
where 
\[
\tilde{K}(x^{i},z^{i};\rho):=\big\{ y\in\R_{+}^{n}:\exists\,\theta\in[0,1]\text{ s.t. }\max_{i}\big|\sqrt{y^{i}}-\theta\sqrt{z^{i}}-(1-\theta)\sqrt{x^{i}}\big|\leq\rho\big\}.
\]
\end{lem}

\begin{proof}
By rescaling $\tilde{X}_{t}=l^{-2}X_{l^{2}t}$ we may prove the lemma
only for $l=1$. 

For $i=1,\dots,n$, set $Y_{t}^{i}=\sqrt{X_{t}^{i}}$ starting from
$\sqrt{z^{i}}$, then on $\{X^{i}>0\}$ it satisfies
\begin{equation}
\md Y_{t}^{i}=\frac{4b_{t}^{i}-|\sigma_{t}^{i}|^{2}}{8Y_{t}^{i}}\vd t+\frac{\sigma^{i}}{2}\vd W_{t};\label{eq:Y}
\end{equation}
 and denote 
\[
\varphi^{i}(t):=\begin{cases}
\sqrt{z^{i}}+\epsilon^{-1}t(\sqrt{x^{i}}-\sqrt{z^{i}}), & t\in[0,\epsilon),\\
\sqrt{x^{i}}, & t\in[\epsilon,\alpha].
\end{cases}
\]

\noindent As we only concern the behavior of $Y^{i}$ before it exits
from $[\varphi^{i}-3c/4,\varphi^{i}+3c/4]$, one can redefine the
drift coefficient of (\ref{eq:Y}) outside this region to make it
bounded by a constant depending only on $c$ and $\bd$. Let $\widehat{Y}^{i}$
denote the solution to the modified SDE that is nondegenerate, we
derive that
\begin{align*}
 & \Prob^{z}\Big[\sup_{\epsilon\le s\le\alpha}\max_{i}|\sqrt{X_{s}^{i}}-\sqrt{x^{i}}|\le\frac{3}{4}c;\,X_{s}\in\tilde{K}(x,z;\frac{3}{4}c)\,\forall s\in[0,\alpha]\Big]\\
=\, & \Prob^{z}\Big[\sup_{\epsilon\le s\le\alpha}\max_{i}|Y_{s}^{i}-\varphi^{i}(s)|\leq\frac{3}{4}c;\,X_{s}\in\tilde{K}(x,z;\frac{3}{4}c)\,\forall s\in[0,\alpha]\Big]\\
\geq\, & \Prob^{z}\Big[\sup_{0\le s\le\alpha}\max_{i}|Y_{s}^{i}-\varphi^{i}(s)|\leq\frac{3}{4}c;\,X_{s}\in\tilde{K}(x,z;\frac{3}{4}c)\,\forall s\in[0,\alpha]\Big]\\
\geq\, & \Prob^{z}\Big[\sup_{0\le s\le\alpha}\max_{i}|Y_{s}^{i}-\varphi^{i}(s)|\leq\frac{3}{4}c\Big]\\
=\, & \Prob^{z}\Big[\sup_{0\le s\le\alpha}\max_{i}|\widehat{Y}_{s}^{i}-\varphi^{i}(s)|\le\frac{3}{4}c\Big].
\end{align*}
Applying \citep[Theorem I.8.5]{bass1998diffusions} to $\widehat{Y}$,
there exists a constant $m_{1}=m_{1}(c,\theta,\alpha,\beta,\bd)>0$
as a lower bound for the last probability. The lemma is proved.
\end{proof}
Applying the above two lemmas we can immediately obtain the following
estimate for $X$, which shows that, with a positive probability,
the components of $X$ starting near boundary leave the boundary rapidly
meanwhile the others still stay away from the boundary.
\begin{defn}
\label{def:regular} A cube $K(x,\rho)$ or a hypercube $Q_{\theta}(t,x,\rho)$
is said to be \emph{regular} if either $x^{i}=0$ or $x^{i}\ge\rho^{2}$
for all $i=1,\dots,n$.
\end{defn}

\begin{prop}
\label{prop:smallcube} For $x_{0}\in\R_{+}^{n}$, assume that Condition
\emph{(}\textbf{\emph{C'}}\emph{)} holds on the regular cube $K(x_{0},1)$.
Let $\beta>1$, $0<c\leq1$, $\alpha>\epsilon>0$ and $r\in[1/2,1)$.
Then, there exists a positive constant $M_{\ref{prop:smallcube}}=M_{\ref{prop:smallcube}}(c,\gamma,\alpha,\beta,r,\bd)$
such that for any cube $K(x,l)\subset K(x_{0},1)$ with $0<cl\le\min_{i}\sqrt{x^{i}}$
and $l<1$ we have
\begin{equation}
\Prob^{y}\big[X_{t}\in K(x,3cl/4),\,t\leq\ext_{Q_{1}(0,x_{0},1)}\big]\geq M_{\ref{prop:smallcube}}\label{eq:l8-1}
\end{equation}
for any $t\in[\epsilon l^{2},\alpha l^{2}]$ and $y\in K(x,\beta l)\cap K(x_{0},r)$.
\end{prop}

\begin{proof}
Let $\tau_{Q_{1}(0,x_{0},1)}=\tau_{Q_{1}(0,x_{0},1)}^{0,y}$ be the
exit time of the process $X$ starting from $(0,y)$. Set $\bar{t}=\frac{\epsilon l^{2}}{4}$
, and let $S$ be a random variable uniformly distributed on $[\bar{t},2\bar{t}]$
and independent of $\mathcal{F}$. We shall prove the lemma by dealing
with $X$ on two time intervals $[0,S]$ and $[S,t]$.

\textit{First, we show that before $2\bar{t}$, }$X$\textit{ leaves
the boundary at a positive probability}. For any $y\in K(x_{0},r)$,
applying assertion (b) of Lemma \ref{lem:pre-est} for $X^{i}$ with
$\xi=\xi(\bar{t},\bd,\frac{1}{4n}),$ we obtain
\begin{equation}
\sum_{\sqrt{x_{0}^{i}}=0}\Prob^{y}\big[X_{S}^{i}\leq\xi\big]\leq n\cdot\frac{1}{4n}=\frac{1}{4}.\label{eq:SmallCube-2}
\end{equation}
 Let $c_{1}$ be a positive number will be determined later. Then
using assertion (a) of Lemma \ref{lem:pre-est} for $X^{i}$ on time
interval $[0,2\bar{t}]$ with $\kappa=\kappa(\frac{1}{4n},c_{1},\bd)$
and 
\begin{equation}
\rho:=\sqrt{2\bar{t}/\kappa}=\sqrt{\frac{\epsilon}{2\kappa}}l,\label{eq:smallcube-2-1}
\end{equation}
 we have 
\begin{align}
 & \Prob^{y}\Big[\sup_{0\leq t\leq2\bar{t}=\kappa\rho^{2}}\big|\sqrt{X_{t}^{i}}-\sqrt{y^{i}}\big|\geq c_{1}\rho,\,i=1,2,\ldots n\Big]\label{eq:smallcube-3}\\
\leq\, & \sum_{i=1}^{n}\Prob^{y}\big[\sup_{0\leq t\leq2\bar{t}}\big|\sqrt{X_{t}^{i}}-\sqrt{y^{i}}\big|\geq c_{1}\rho\big]\leq\frac{1}{4}.\nonumber 
\end{align}
We require $c_{1}$ satisfying 
\begin{equation}
c_{1}\rho\leq\frac{1-r}{2},\label{eq:smal-cub-3-3}
\end{equation}
then, keeping $y\in K(x_{0},r)$ in mind, the relation $\sup_{t\in[0,2\bar{t}],i=1,\ldots,n}\big|\sqrt{X_{t}^{i}}-\sqrt{y^{i}}\big|\leq c_{1}\rho$
implies $X_{s}\in K(x_{0},1)$ for every $s\in[0,S]$ on events $\{S\leq\tau_{Q_{1}(0,x_{0},1)}\}$.
Then it follows by (\ref{eq:SmallCube-2}) and (\ref{eq:smallcube-3})
that, for any $y\in K(x_{0},r)$,
\begin{align}
 & \Prob^{y}\big[X_{S}\in K(y,c_{1}\rho),\,X_{S}^{i}>\xi\ \text{ if}\ \sqrt{x_{0}^{i}}\leq1,\,S\leq\tau_{Q_{1}(0,x_{0},1)}\big]\label{eq:smallcube-4}\\
= & \Prob^{y}\big[\widehat{X}_{S}\in K(y,c_{1}\rho),\,X_{S}^{i}>\xi\ \text{ if}\ \sqrt{x_{0}^{i}}\leq1,\,S\leq\tau_{Q_{1}(0,x_{0},1)}\big]\nonumber \\
\geq & \Prob^{y}\big[\big|\sqrt{\widehat{X}_{S}^{i}}-\sqrt{y^{i}}\big|\leq c_{1}\rho,\,X_{S}^{i}>\xi\ \text{ if}\ \sqrt{x_{0}^{i}}\leq1,\,i=1,2,\ldots,n;\,\big]\nonumber \\
\geq & 1-\Prob^{y}\big[X_{S}^{i}>\xi\ \text{ if}\ \sqrt{x_{0}^{i}}\leq1,\,i=1,2,\ldots,n\big]\nonumber \\
 & -\Prob^{y}\big[\big|\sqrt{X_{S}^{i}}-\sqrt{y^{i}}\big|\geq c_{1}\rho,\,i=1,2,\ldots,n\big]\nonumber \\
\geq & 1-\frac{1}{4}-\frac{1}{4}=\frac{1}{2}.\nonumber 
\end{align}

\textit{Second, we show that $X$ hits any small cube in a positive
probability at time $t\in[\epsilon l^{2},\alpha l^{2}]$}. For every
$s\in[\bar{t},2\bar{t}]$, $z\in K(y,c_{1}\rho)$ with $z^{i}>\xi$
if $\sqrt{x_{0}^{i}}\leq1$, by (\ref{eq:smal-cub-3-3}), if $\sqrt{x^{i}}\geq1$
\begin{align*}
|\sqrt{z^{i}}-\sqrt{x_{0}^{i}}|\leq & |\sqrt{z^{i}}-\sqrt{y^{i}}|+|\sqrt{y^{i}}-\sqrt{x_{0}^{i}}|\\
\leq & c_{1}\rho+r\leq\frac{1-r}{2}+r\\
= & \frac{1+r}{2},
\end{align*}
then $\sqrt{z^{i}}\geq\sqrt{x_{0}^{i}}-\frac{1+r}{2}\geq1-\frac{1+r}{2}=\frac{1-r}{2}$.
Besides, $\sqrt{z^{i}}\geq\sqrt{\xi}$ if $\sqrt{x_{0}^{i}}=0$. So
$\sqrt{z^{i}}>\min\{\sqrt{\xi},\frac{1-r}{2}\}$ for $i=1,\ldots,n$.

In order to ensure 
\[
c_{1}\rho\leq\min\{cl,\sqrt{\xi},\sqrt{x^{1}},\dots,\sqrt{x^{n}}\}
\]
with constraint (\ref{eq:smal-cub-3-3}), we take
\begin{equation}
c_{1}=\sqrt{\frac{2\kappa}{\epsilon}}\min\{\sqrt{\xi},\frac{1-r}{2},c\}.\label{eq:small-cube-4-1}
\end{equation}
So one has
\begin{align*}
\max_{i}\big|\sqrt{z^{i}}-\sqrt{x^{i}}\big|\leq & \max_{i}\big|\sqrt{z^{i}}-\sqrt{y^{i}}\big|+\max_{i}\big|\sqrt{y^{i}}-\sqrt{x^{i}}\big|\\
\leq & c_{1}\rho+\beta l\\
\leq & \big(\sqrt{\epsilon/(2\kappa)}+\beta\big)l.
\end{align*}
Applying Lemma \ref{lem:supp} on the period $[s,s+(\alpha-\epsilon/4)l^{2}]$
and noticing that $[\epsilon l^{2},\alpha l^{2}]\subset[s+\frac{\epsilon}{4}l^{2},s+(\alpha-\epsilon/4)l^{2}]$,
one can derive that, for any $t\in[\epsilon l^{2},\alpha l^{2}]$,
\begin{align}
 & \Prob^{s,z}\Big[X_{t}\in K(x,\frac{3}{4}c_{1}\rho),\ t\leq\ext_{Q_{1}(0,x_{0},1)}\Big]\label{eq:smallcube-5}\\
\geq & \Prob^{s,z}\Big[X_{t}\in K(x,\frac{3}{4}c_{1}\sqrt{\frac{\epsilon}{2\kappa}}l),\ t\leq\ext_{Q_{1}(0,x_{0},1)},\ \forall t\in[\epsilon l^{2},\alpha l^{2}]\Big]\nonumber \\
\geq & \Prob^{s,z}\bigg[\begin{array}{l}
X_{t}\in K\big(x,\frac{3}{4}c_{1}\sqrt{\frac{\epsilon}{2\kappa}}l\big)\ \forall t\in[s+\frac{\epsilon}{4}l^{2},s+(\alpha-\epsilon/4)l^{2}];\\
X_{t}\in K\Big(x,z;\frac{3}{4}c_{1}\sqrt{\frac{\epsilon}{2\kappa}}l\big)\Big)\ \forall t\in[s,s+(\alpha-\epsilon/4)l^{2}]
\end{array}\bigg]\nonumber \\
\geq & m_{1}(c_{1}\sqrt{\frac{\epsilon}{2\kappa}},\epsilon/4,\alpha-\varepsilon/4,\sqrt{\epsilon/(2\kappa)}+\beta)=:M_{0}.\nonumber 
\end{align}
Combining (\ref{eq:smallcube-5}) and the strong Markov property of
$X$, we obtain that for any $y\in K(x,\beta l)\cap K(x_{0},r)$,
\begin{align*}
 & \Prob^{y}[X_{t}\in K(x,3cl/4),\,t\leq\ext_{Q_{1}(0,x_{0},1)}]\\
\geq & \E^{y}\Big[\Prob^{S,X_{S}}\big[X_{t}\in K(x,\frac{3}{4}c_{1}\rho),\,t\leq\ext_{Q_{1}(0,x_{0},1)}\Big];\\
 & \qquad\qquad X_{S}\in K(y,\frac{3}{4}c_{1}\rho),\,X_{S}^{i}>\xi\ \text{if}\ \sqrt{x_{0}^{i}}\leq1,\,S\leq\ext_{Q_{1}(0,x_{0},1)}\Big]\\
\ge & M_{0}\Prob^{y}\big[\,X_{S}\in K(y,\frac{3}{4}c_{1}\rho),\,X_{S}^{i}>\xi\ \text{if}\ \sqrt{x_{0}^{i}}\leq1,\,S\leq\ext_{Q_{1}(0,x_{0},1)}\big]\\
\ge & \frac{1}{2}M_{0}=:M_{\ref{prop:smallcube}}(c,\epsilon,\alpha,\beta,r,\bd)\quad\text{using \eqref{eq:smallcube-4}}.
\end{align*}
The proof is complete.
\end{proof}
The following corollary gives a lower bound of the probability of
$X$ hitting any compact subset of a cube.
\begin{cor}
\label{cor:hit-prob} Under the assumption of Proposition \ref{prop:smallcube},
there exists positive constant $M_{\ref{cor:hit-prob}}=M_{\ref{cor:hit-prob}}(c,\epsilon,\alpha,\beta,r,\bd)$
such that for any cube $K(x,l)\subset K(x_{0},1)$ we have
\begin{equation}
\Prob^{y}\big[X_{t}\in K(x,3cl/4),\,t\leq\ext_{Q_{1}(0,x_{0},1)}\big]\geq M_{\ref{cor:hit-prob}}\label{eq:hit-prob-1}
\end{equation}
for any $l<1$, $t\in[\epsilon l^{2},\alpha l^{2}]$ and $y\in K(x,\beta l)\cap K(x_{0},r)$.
\end{cor}

\begin{proof}
To apply Proposition \ref{prop:smallcube}, we turn to estimate the
hitting probability of subset of $K(x,3cl/4)$ with a distance away
from $\partial\R_{+}^{n}$. Define 
\[
\sqrt{\hat{x}^{i}}:=\begin{cases}
\sqrt{x^{i}}, & {\rm if}\,\sqrt{x^{i}}>cl;\\
\sqrt{x^{i}}+3cl/8, & {\rm if}\,\sqrt{x^{i}}>cl.
\end{cases}
\]
Let $\hat{c}=3c/8$, then $K(\hat{x},3\hat{c}l/4)\subset K(x,3cl/4)$
and $\min_{i}\sqrt{\hat{x}^{i}}\geq\hat{c}l$. Then by Proposition
\ref{prop:smallcube} we have 
\begin{align*}
 & \Prob^{y}\big[X_{t}\in K(x,3cl/4),\,t\leq\ext_{Q_{1}(0,x_{0},1)}\big]\\
\geq & \Prob^{y}\big[X_{t}\in K(\hat{x},3\hat{c}l/4),\,t\leq\ext_{Q_{1}(0,x_{0},1)}\big]\\
\geq & M_{\ref{prop:smallcube}}(\hat{c},\epsilon,\alpha,\beta,r,\bd)=:M_{\ref{cor:hit-prob}}.
\end{align*}
The corollary is proved. 
\end{proof}

\subsection{A measure theory lemma}

As in Krylov and Safonov's original argument, we need a measure theory
lemma concerning a Calderón\textendash Zygmund-type decomposition
for anisotropic hypercubes defined by (\ref{eq:cube}).

In this subsection, we denote $Q:=Q_{\theta}(0,x_{0},1)$ and assume
$Q$ is regular (see Definition \ref{def:regular} above). 

The purpose of the following lemma is to decompose $Q$ into the union
of smaller sub-hypercubes according to the proportion (of the sub-hypercube)
occupied by a closed set $\Gamma\subset Q$. Given $\mu,\eta\in(0,1)$
we define two sets 
\begin{align*}
D_{1} & =\bigcup\big\{ Q\cap[(t-3\theta\rho^{2},t+4\theta\rho^{2})\times K(t,x,3\rho)]:\\
 & \qquad\qquad\tilde{Q}:=Q_{\theta}(t,x,\rho)\subset Q,\,|\Gamma\cap\tilde{Q}|\geq\mu|\tilde{Q}|,\,{\rm and}\,\tilde{Q}\,{\rm is}\,{\rm regular}\big\},\\
D_{2} & =\bigcup\big\{(t-\theta\rho^{2}-4\theta\rho^{2}/\eta,t-\theta\rho^{2})\times[K(t,x,3\rho)\cap K(0,x_{0},1)]:\\
 & \qquad\qquad\tilde{Q}:=Q_{\theta}(t,x,\rho)\subset Q,\,|\Gamma\cap\tilde{Q}|\geq\mu|\tilde{Q}|,\,{\rm and}\,\tilde{Q}\,{\rm is}\,{\rm regular}\big\}.
\end{align*}

\begin{lem}
\label{lem:measure th}\emph{(a)} $|\Gamma|\leq\mu|Q|$ implies $|\Gamma|\leq\mu|D_{1}|$.

\emph{(b)} $|D_{1}|\leq(1+\eta)|D_{2}|.$

\emph{(c)} For $0<\mu'<\mu<1$, if $|\Gamma\cap Q|\geq\mu'|Q|$, and
let $\eta=\frac{1}{\sqrt{\mu}}-1$, then one has that either 
\[
|D_{2}\cap Q|\geq\mu^{-\frac{1}{4}}\mu'|Q|,
\]
or there exits a regular hypercube $Q_{\theta}(\check{t},\check{z},\check{\rho})\subset Q$
with $\check{\rho}\geq\frac{1}{4}(1-\sqrt{\mu})\sqrt{\mu'}$ such
that 
\begin{equation}
|Q_{\theta}(\check{t},\check{z},\check{\rho})\cap\Gamma|\geq\mu|Q_{\theta}(\check{t},\check{z},\check{\rho})|.\label{eq:measure-0}
\end{equation}
\end{lem}

\begin{proof}
(a) We divide $Q$ in to a union of smaller hypercubes with disjoint
interiors:
\begin{itemize}
\item along $t$-axis: partition $Q$ to nine equal parts by hyperplanes
$t=\theta i/3^{2},\:i=1,2,\ldots,8$;
\item along $x$- axises: for $i=1,2,\ldots,n$,
\begin{itemize}
\item if $\sqrt{x_{0}^{i}}\geq1$, we partition $Q$ by hyperplanes $\sqrt{x^{i}}=\sqrt{x_{0}^{i}}-\frac{1}{3}$
and $\sqrt{x^{i}}=\sqrt{x_{0}^{i}}+\frac{1}{3}$,
\item if $\sqrt{x_{0}^{i}}=0$, we partition $Q$ by hyperplane $\sqrt{x^{i}}=\frac{1}{3}$.
\end{itemize}
\end{itemize}
Obviously, every sub-hypercube is regular and of form $Q_{\theta}(t,x,1/3)$
with some $(t,x)\in Q$. We denote these sub-hypercube by $Q_{j_{1}}$.

We construct $n$-level sub-hypercubes by induction. Suppose $(n-1)$-level
regular sub-hypercubes are defined. Then we partition an $(n-1)$-level
sub-hypercube $Q_{j_{1}j_{2}\ldots j_{n-1}}=Q_{\theta}(\hat{t},\hat{x},\frac{1}{3^{n-1}})$
into smaller hypercubes in a similar way for $Q$: 
\begin{itemize}
\item along $t$-axis: partition $Q_{j_{1}j_{2}\ldots j_{n-1}}$ to nine
equal parts by hyperplanes $t=\hat{t}+\theta i/3^{n+1},\:i=1,2,\ldots,8$;
\item along $x$-axises: for $i=1,2,\ldots,n$,
\begin{itemize}
\item if $\sqrt{\hat{x}^{i}}\geq1$, we partition $Q_{j_{1}j_{2}\ldots j_{n-1}}$
by hyperplanes $\sqrt{x^{i}}=\sqrt{\hat{x}^{i}}-\frac{1}{3^{n}}$
and $\sqrt{x^{i}}=\sqrt{\hat{x}^{i}}+\frac{1}{3^{n}}$,
\item if $\sqrt{\hat{x}^{i}}=0$, we partition $Q_{j_{1}j_{2}\ldots j_{n-1}}$
by hyperplanes $\sqrt{x^{i}}=\frac{1}{3^{n}}$.
\end{itemize}
\end{itemize}
Every sub-hypercube obtained in this step, labeled with $Q_{j_{1}j_{2}\ldots j_{n-1}j_{n}}$,
is also regular and of form $Q_{\theta}(t,x,\frac{1}{3^{n}})$ with
some $(t,x)\in Q_{j_{1}j_{2}\ldots j_{n-1}}$. We remark that the
number of $j_{n}$'s values may differ from different $Q_{j_{1}j_{2}\ldots j_{n-1}j_{n}}$.

We denote by $\mathscr{S}$ a family of all sub-hypercubes satisfying
the following conditions: i) the sub-hypercube, say $Q_{j_{1}j_{2}\ldots j_{n-1}}$
with some $n$, satisfies 
\begin{equation}
|Q_{j_{1}j_{2}\ldots j_{n-1}}\cap\Gamma|<\mu|Q_{j_{1}j_{2}\ldots j_{n-1}}|,\label{eq:measure-1}
\end{equation}
and ii) there is at least one $Q_{j_{1}j_{2}\ldots j_{n-1}j_{n}}$
obtained from $Q_{j_{1}j_{2}\ldots j_{n-1}}$ such that
\[
|Q_{j_{1}j_{2}\ldots j_{n-1}j_{n}}\cap\Gamma|\geq\mu|Q_{j_{1}j_{2}\ldots j_{n-1}j_{n}}|.
\]
From the definition of $D_{1}$ it is easily known that
\[
\tilde{\Gamma}:=\cup_{\tilde{Q}\in\mathscr{S}}\tilde{Q}\subset D_{1},
\]
and by the relation (\ref{eq:measure-1}),
\[
|\Gamma\cap\tilde{\Gamma}|=\sum_{\tilde{Q}\in\mathscr{S}}|\Gamma\cap\tilde{Q}|<\mu\sum_{\tilde{Q}\in\mathscr{S}}|\tilde{Q}|=\mu|\tilde{\Gamma}|\leq\mu|D_{1}|.
\]
If one can show that $|\Gamma\backslash\tilde{\Gamma}|=0$, then Assertion
(a) is valid because
\[
|\Gamma|\leq|\Gamma\cap\tilde{\Gamma}|+|\Gamma\backslash\tilde{\Gamma}|\leq\mu|D_{1}|.
\]

Now we prove $|\Gamma\backslash\tilde{\Gamma}|=0$ by Lebesgue's theorem
(seeing \citep[Theorem 7.10]{rudin1987Real}). Notice that every point
in $\Gamma\backslash\tilde{\Gamma}$ is the limit of a sequence of
sub-hypercubes $\tilde{Q}^{k}$ with radius $3^{-k}$ and $|\Gamma\cap\tilde{Q}^{k}|<\mu|\tilde{Q}^{k}|$,
$k=1,2,\ldots$. Applying Lebesgue's theorem to the function $\mathbf{1}_{\Gamma}(\cdot)$,
one knows
\[
\mathbf{1}_{\Gamma}\leq\mu\quad\text{a.e. on }\:\Gamma\backslash\tilde{\Gamma}.
\]
 This along with $\mu<1$ yields $|\Gamma\backslash\tilde{\Gamma}|=0$.
Hence, Assertion (a) is proved. \medskip{}

The proof of Assertion (b) is quite similar to that of Lemma 2.3 in
\citep{Krylov1981A}, so we omit it here. Next we give a proof of
Assertion (c); a similar result can be found in the textbook \citep[Lemma 2.4 , Ch 7]{chen2003second}
in Chinese.

We may assume $|\Gamma|\leq\mu|Q|$ without loss of generality, otherwise
the relation (\ref{eq:measure-0}) already holds for $Q$ itself.
We discuss the following two cases:

\textit{(1) }$|D_{2}\backslash Q|\leq\mu^{-\frac{1}{4}}\big(\mu^{-\frac{1}{4}}-1\big)\mu'|Q|$.

Using assertion (b), we have
\begin{align*}
|D_{2}\cap Q| & =|D_{2}|-|D_{2}\backslash Q|\\
 & \geq\frac{1}{1+\eta}|D^{1}|-\mu^{-\frac{1}{4}}\big(\mu^{-\frac{1}{4}}-1\big)\mu'|Q|.
\end{align*}
 It follows from assertion (a) that 
\begin{align*}
|D_{2}\cap Q| & \geq\frac{1}{(1+\eta)\mu}|\Gamma|-\mu^{-\frac{1}{4}}\big(\mu^{-\frac{1}{4}}-1\big)\mu'|Q|\\
 & \geq\frac{\mu'}{\sqrt{\mu}}|Q|-\mu^{-\frac{1}{4}}\big(\mu^{-\frac{1}{4}}-1\big)\mu'|Q|\\
 & =\mu^{-\frac{1}{4}}\mu'|Q|.
\end{align*}

\textit{(2) }$|D_{2}\backslash Q|>\mu^{-\frac{1}{4}}\big(\mu^{-\frac{1}{4}}-1\big)\mu'|Q|$.

By the definition of $D_{2}$, there exists $Q_{\theta}(\check{t},\check{z},\check{\rho})\subset Q$
satisfying $|Q_{\theta}(\check{t},\check{z},\check{\rho})\cap\Gamma|\geq\mu|Q_{\theta}(\check{t},\check{z},\check{\rho})|$
and $4\check{\rho}^{2}/\eta\geq\mu^{-\frac{1}{4}}\big(\mu^{-\frac{1}{4}}-1\big)\mu'$,
which implies $\check{\rho}\geq\frac{1}{4}(1-\sqrt{\mu})\sqrt{\mu'}$.
\end{proof}

\section{Hitting probability of large sets}

We now prove Theorem \ref{thm:main2} when $|\Gamma\cap Q|/|Q|$ is
large enough. 
\begin{prop}
\label{prop:large target-1} Let Condition \emph{(}\textbf{\emph{C'}}\emph{)}
hold on $K(x_{0},\rho)$ with $x_{0}\in\R_{+}^{n}$ and $\rho<1$.
For $\theta\in(0,1)$, there exist $\mu_{0}=\mu_{0}(\theta)\in(0,1)$
and $\varepsilon=\varepsilon(\mu_{0})>0$ such that for any $x\in K(x_{0},3\rho/4)$
and any closed set $\Gamma\subset Q=Q_{\theta}(t_{0},x_{0},\rho)$
satisfying $|\Gamma|\geq\mu_{0}|Q|$ we have that
\begin{equation}
\Prob^{t_{0},x}[\hit_{\Gamma}\leq\ext_{Q}]\geq\varepsilon(\mu_{0}),\label{eq:3-0-0}
\end{equation}
where $(t_{0},x_{0})\in[0,\infty)\times\mathbb{R}_{+}^{n}$ and $\rho\in(0,1]$
are arbitrarily given.
\end{prop}

\begin{rem*}
The constants $\mu_{0}$ and $\varepsilon_{0}$ actually depend additionally
on $n$ and $\bd$. Here we only emphasize their dependence on $\theta$
for convenience.
\end{rem*}
\begin{proof}
According to Remark \ref{rem:secaling} (2) we may assume $t_{0}=0$
and $\rho=1$ without loss of generality.

Denote $Q=Q_{\theta}(0,x_{0},1)$ and $\mu=|\Gamma\cap Q|/|Q|$. Let
$\delta\leq1/8$ be a constant specified later in (\ref{eq:large-11}),
and denote
\[
Q^{\delta}:=\{(s,y)\in Q\,|\,\sqrt{y^{i}}\geq\delta,i=1,\cdots,n\}.
\]
We consider two cases in terms of the location of initial point $x$. 

\medskip

\textbf{Case 1}: $x\in K(x_{0},7/8)\cap[4\delta^{2},\infty)^{n}$.

Applying Lemma \ref{lem:pre-est}(a) to $X^{i}$ ($i=1,\dots,n$)
with $\rho=\delta$, there is a small positive number $\kappa_{1}=\kappa(\frac{1}{2n},1,\bd)>0$
such that 
\[
\Prob^{x}\Big[\sup_{0\le t\le\kappa_{1}\delta^{2}}\sup_{i=1,2,\ldots,n}|\sqrt{X_{t}^{i}}-\sqrt{x^{i}}|\ge\delta\Big]\le\frac{1}{2}.
\]
Since $|\sqrt{y}-\sqrt{x}|\leq\delta$ implies $y\in K(x_{0},1)$,
if we require 
\begin{equation}
\kappa_{1}\delta^{2}\leq\theta,\label{eq:delta-1}
\end{equation}
then 
\begin{equation}
\E^{x}[\ext_{Q^{\delta}}]\geq\kappa_{1}\delta^{2}\Prob^{x}\Big[\sup_{0\le t\le\kappa_{1}\delta^{2}}\sup_{i=1,2,\ldots,n}|\sqrt{X_{t}^{i}}-\sqrt{x^{i}}|<\delta\Big]\geq\frac{\kappa_{1}\delta^{2}}{2}.\label{eq:3-0}
\end{equation}
So in this case we choose
\begin{equation}
\delta\leq\min\big\{\sqrt{\theta/\kappa_{1}},1/8\big\}<1.\label{eq:large-2}
\end{equation}

Now we normalize the process $X$ as follows:
\[
\hat{X}_{t}^{i}:=X_{\theta t}^{i}/E^{i},\qquad i=1,2,\ldots,n,
\]
where 
\[
E^{i}:=\begin{cases}
(\sqrt{x_{0}^{i}}+1)^{2}-(\sqrt{x_{0}^{i}}-1)^{2}=4\sqrt{x_{0}^{i}}, & {\rm if}\,(\sqrt{x_{0}^{i}}-1)^{+}\geq\delta,\\
(\sqrt{x_{0}^{i}}+1)^{2}-\delta^{2}, & {\rm if}\,(\sqrt{x_{0}^{i}}-1)^{+}<\delta
\end{cases}
\]
is the width of $Q^{\delta}$ along the $i$-th coordinate direction.
Correspondingly, we do a change of variables $\hat{x}:=(x^{i}/E^{i})_{i=1}^{n}$.
Evidently, $\hat{X}$ satisfies SDE (\ref{eq:SDE}) with $\hat{b}^{i}(x):=(E^{i})^{-1}\theta b^{i}(E^{i}x)$
and $\hat{\sigma}^{ik}(x):=(E^{i})^{-\frac{1}{2}}\theta^{\frac{1}{2}}\sigma^{ik}(E^{i}x)$
instead of $b^{i}$ and $\sigma^{ik}$, respectively, for $i=1,\ldots,n$
and $k=1,2,\ldots$, and with $\hat{W}_{t}=\theta^{-\frac{1}{2}}W_{\theta t}$
instead of $W_{t}$. For any set $G\subset[0,\infty)\times\R_{+}^{n}$,
denote
\[
\hat{G}:=\{(\theta^{-1}t,\hat{x}):\hat{x}^{i}=x^{i}/E^{i},\,(t,x)\in G\}.
\]
Then one has
\[
\theta^{-1}\ext_{G}=\ext_{\hat{G}}:=\inf\{t\ge0:\hat{X}_{s}^{0,\hat{x}}\in\hat{G}\}.
\]
Moreover, a simple computation shows that, for any $\hat{x}\in\widehat{Q^{\delta}}$,
\begin{align}
|\hat{b}^{i}(\hat{x})| & \le2\bd,\nonumber \\
\widehat{A}(\hat{x}) & \coloneqq\big(\langle\hat{\sigma}^{i},\hat{\sigma}^{j}\rangle\sqrt{\hat{x}^{i}\hat{x}^{j}}\big)_{i,j=1}^{n}\label{eq:large-3}\\
 & =\Big(\theta\langle\sigma^{i},\sigma^{j}\rangle\frac{\sqrt{x^{i}x^{j}}}{E^{i}E^{j}}\Big)_{i,j=1}^{n}>\frac{\theta\bd^{-1}\delta^{2}}{64}I_{n}.\nonumber 
\end{align}

Now applying \citep[Theorem 2.2.2]{Krylov1980Controlled} to $\hat{X}_{t}$
on $\widehat{Q^{\delta}}$ with $F(c,a)=c$, $c_{t}=2\bd$ and $g=\mathbf{1}_{\widehat{Q^{\delta}\backslash\Gamma}}$,
we have
\begin{align}
 & \E^{\hat{x}}\int_{0}^{\ext_{\widehat{Q^{\delta}}}}\exp(-2\bd s)\big({\rm det}\,\widehat{A}\big)^{\frac{1}{n+1}}\mathbf{1}_{\widehat{Q^{\delta}\backslash\Gamma}}(s,\hat{X}_{s})\vd s\label{eq:3-1}\\
 & \leq C_{0}\|\mathbf{1}_{\widehat{Q^{\delta}\backslash\Gamma}}\|_{L^{n+1}}\leq C_{0}\|\mathbf{1}_{\widehat{Q\backslash\Gamma}}\|_{L^{n+1}}\le C_{0}[(1-\mu)\big|\hat{Q}|]^{\frac{1}{n+1}}\nonumber \\
 & \leq C_{0}\big(2^{n}\big|\widehat{Q^{\delta}}\big|\big)^{\frac{1}{n+1}}(1-\mu)^{\frac{1}{n+1}}=C_{0}2^{\frac{n}{n+1}}(1-\mu)^{\frac{1}{n+1}},\nonumber 
\end{align}
where the constant $C_{0}=C_{0}(n)>1$. (\ref{eq:large-3}) shows
for any $s\in[0,\ext_{\widehat{Q^{\delta}}}(\omega)]$,
\[
{\rm det}\big(\widehat{A}(\hat{X}_{s}(\omega))\big)\geq\Big(\frac{\theta\bd}{64}\Big)^{n}\delta^{2n},
\]
which combining with (\ref{eq:3-1}) and 
\begin{align*}
\E^{x}\big[\ext_{Q^{\delta}};\hit_{\Gamma}\geq\ext_{Q^{\delta}}\big] & \leq\E^{x}\int_{0}^{\ext_{Q^{\delta}}}\mathbf{1}_{Q^{\delta}\backslash\Gamma}(s,X_{s})\vd s\\
 & =\E^{\hat{x}}\int_{0}^{\ext_{\widehat{Q^{\delta}}}}\mathbf{1}_{\widehat{Q^{\delta}\backslash\Gamma}}(s,\hat{X}_{s})\vd s
\end{align*}
implies that
\[
\me^{-2\bd}\Big(\frac{\theta}{64\bd}\Big)^{\frac{n}{n+1}}\delta^{\frac{2n}{n+1}}\E^{x}\big[\ext_{Q^{\delta}};\hit_{\Gamma}\geq\ext_{Q^{\delta}}\big]\leq C_{0}2^{\frac{n}{n+1}}(1-\mu)^{\frac{1}{n+1}}.
\]
If choosing $\mu\in(0,1)$ to satisfy 
\begin{align}
(1-\mu)\delta^{-(4n+2)} & \leq C_{0}^{-(n+1)}(128\bd)^{-n}\me^{-2(n+1)\bd}\big(\frac{\kappa_{1}}{4}\big)^{n+1}\theta^{n}=:M(\theta)\label{eq:9-4}
\end{align}
we have that
\[
\E^{x}\big[\ext_{Q^{\delta}};\hit_{\Gamma}\geq\ext_{\hat{Q}^{\delta}}\big]\leq\frac{\kappa_{1}}{4}\delta^{2}.
\]
Noticing that $\ext_{Q^{\delta}}\le1$ and (\ref{eq:3-0}), we compute
that
\begin{align*}
\frac{\kappa_{1}}{2}\delta^{2} & \le\E^{x}\big[\ext_{Q^{\delta}}\big]\\
 & =\E^{x}\big[\ext_{Q^{\delta}};\hit_{\Gamma}<\ext_{Q^{\delta}}\big]+\E^{x}\big[\ext_{Q^{\delta}};\hit_{\Gamma}\geq\ext_{Q^{\delta}}\big]\\
 & \le\Prob^{x}[\hit_{\Gamma}<\ext_{Q^{\delta}}]+\frac{\kappa_{1}}{4}\delta^{2}.
\end{align*}
Therefore, we gain that
\[
\Prob^{x}[\hit_{\Gamma}<\ext_{Q}]\geq\Prob^{x}[\hit_{\Gamma}<\ext_{Q^{\delta}}]\geq\frac{\kappa_{1}}{4}\delta^{2},
\]
provided $|\Gamma|\geq\mu|Q|$ with $\mu$ satisfying (\ref{eq:9-4}).

\medskip

\textbf{Case 2}: $x\in K(x_{0},3/4)$.

The idea is to prove that $X$ will enter $K(x_{0},7/8)\cap[4\delta^{2},\infty)^{n}$
in a short time before it leaves $K(x_{0},1)$. Then one can make
use of the result in Case 1 to estimate the hitting probability.

Letting $l=2\delta$, one can choose $z\in K(x_{0},1)$ satisfying
\begin{align*}
 & K(z,2l)\subset K(x_{0},1),\\
 & 2l\leq\min_{i}\sqrt{z^{i}},\\
{\rm and}\, & \max_{i}\big|\sqrt{x^{i}}-\sqrt{z^{i}}\big|\leq2l.
\end{align*}
From Lemma (\ref{prop:smallcube}), there is a constant $M_{\ref{prop:smallcube}}=M_{\ref{prop:smallcube}}(c=1,\epsilon=\theta,\alpha=1,\beta=2,r=\frac{3}{4},\bd)$,
such that 
\begin{equation}
\Prob^{x}\big[X_{\theta l^{2}}\in K(z,\frac{3l}{4});\,X_{t}\in K(x_{0},1)\,\forall t\in[0,\theta l^{2}]\big]\geq M_{\ref{prop:smallcube}}.\label{eq:large-7}
\end{equation}
Obviously, $K(z,\frac{3l}{4})\subset K(x,\frac{7}{8})\cap[4\delta^{2},\infty)^{n}$.

Now we apply the result obtained in Case 1 with $\tilde{Q}:=Q_{\theta(1-l^{2})}(\theta l^{2},x_{0},1)$
instead of $Q=Q_{\theta}(0,x_{0},1)$. Then, if $\Gamma$ satisfies
\begin{align}
|\Gamma\cap\tilde{Q}| & \geq\big[1-M\big(\theta(1-l^{2})\big)\delta^{4n+2}\big]|\tilde{Q}|\label{eq:9-9}
\end{align}
(where $M(\cdot)$ is defined in (\ref{eq:9-4})), one has
\begin{equation}
\Prob^{\theta l^{2},z}\big[\hit_{\Gamma}\leq\ext_{\tilde{Q}}\big]\geq\frac{\kappa_{1}}{4}\delta^{2}.\label{eq:9-10}
\end{equation}
Then by (\ref{eq:large-7}) and (\ref{eq:9-10}), we derive that
\begin{align*}
 & \Prob^{x}\big[\hit_{\Gamma}<\ext_{Q}]\\
 & \ge\Prob^{x}\big[\hit_{\Gamma}<\ext_{Q};\big\{ X_{\theta l^{2}}\in[4\delta^{2},\infty)^{n}\cap K(x_{0},7/8)\big\}\big]\\
 & =\E^{x}\Big[\Prob^{(\theta l^{2},X_{\theta l^{2}})}\big[\hit_{\Gamma}\leq\ext_{\tilde{Q}}\big];\big\{ X_{\theta l^{2}}\in[4\delta^{2},\infty)^{n}\cap K(x_{0},7/8)\big\}\Big]\\
 & \ge\frac{\kappa_{1}}{4}\delta^{2}\Prob^{x}\big\{ X_{\theta l^{2}}\in[4\delta^{2},\infty)^{n}\cap K(x_{0},7/8)\big\}\\
 & \geq\frac{\kappa_{1}}{4}\delta^{2}M_{\ref{prop:smallcube}}\\
 & \eqqcolon\eps.
\end{align*}
Due to the change of parameters from $Q_{\theta}(0,x_{0},1)$ to $Q_{\theta(1-l^{2})}(\theta l^{2},x_{0},1)$,
we should update the choice of the constant $\delta$:
\begin{equation}
\delta=\min\big\{\sqrt{\theta/(\kappa_{1}+4)},1/8\big\}\label{eq:large-11}
\end{equation}
to ensure the relation $\kappa_{1}\delta^{2}\leq\theta(1-l^{2})=\theta(1-4\delta^{2})$,
corresponding to (\ref{eq:delta-1}).

To conclude the proof, it suffices to choose a proper $\mu\in(0,1)$
so that the condition (\ref{eq:9-9}) is satisfied. Using the condition
$|\Gamma\cap Q|\geq\mu|Q|$, we compute that
\begin{align*}
\frac{|\Gamma\cap\tilde{Q}|}{|\tilde{Q}|}= & \frac{|\Gamma\cap Q|-\big|\Gamma\cap(Q-\tilde{Q})\big|}{|Q|-|Q-\tilde{Q}|}\\
\geq & \frac{\mu|Q|-l^{2}|Q|}{|Q|-l^{2}|Q|}\geq\frac{\mu-l^{2}}{1-l^{2}}.
\end{align*}
So the condition (\ref{eq:9-9}) is satisfied if 
\[
\frac{\mu-l^{2}}{1-l^{2}}=1-M\big(\theta(1-l^{2})\big)\delta^{4n+2},
\]
that is, 
\[
\mu=\mu_{0}:=1-(1-l^{2})M\big(\theta(1-l^{2})\big)\delta^{4n+2}\in(0,1).
\]
The proof is complete.
\end{proof}

\section{\label{sec:small}Proof of Theorem \ref{thm:main2}}

In terms of rescaling and translation (see Remark \ref{rem:secaling}
above), we may assume $\rho=1$ and $t=0$. Fix $\theta\in(0,1]$
and denote $Q:=Q_{\theta}(0,x_{0},1)$. 

\subsection{When $Q$ is regular}

In this case we shall prove the assertion of Theorem \ref{thm:main2}
for any initial point $x\in K(x_{0},3/4)$ instead of $x\in K(x_{0},1/6)$. 

Now we define a non-decreasing function $\eps(\cdot):(0,1)\to[0,1]$
as
\begin{align}
\eps(\mu)= & \inf\Big\{\Prob^{x}[\hit_{\Gamma}<\ext_{Q}]\big|x_{0}\in\R_{+}^{n},\ x\in K(x_{0},3\rho/4),\,\label{eq:reg-1}\\
 & \qquad\qquad\tilde{Q}:=Q_{\theta}(0,x_{0},\rho)\:{\rm is}\:{\rm regular},\:\Gamma\subset\tilde{Q},\ |\Gamma|>\mu|\tilde{Q}|,\:\rho\in(0,1]\Big\},\nonumber 
\end{align}
and denote
\[
\underbar{\ensuremath{\mu}}:=\inf\{\mu:\eps(\mu)>0\}.
\]
Obviously, $\underbar{\ensuremath{\mu}}\le\mu_{0}$ where $\mu_{0}$
is the constant determined by Proposition \ref{prop:large target-1}.
If $\underbar{\ensuremath{\mu}}=0$, Theorem \ref{thm:main2} is automatically
concluded. So we suppose $\underbar{\ensuremath{\mu}}>0$ and aim
to deduce a contradiction. 

Define
\[
\left\{ \begin{aligned}q & :=\min\big\{(\mu_{0}/\underbar{\ensuremath{\mu}})^{\frac{1}{2}},\,\mu_{0}^{-\frac{1}{12}}\big\}>1,\\
d_{1} & :=\frac{1}{2}\vee(1+q\underbar{\ensuremath{\mu}}-q^{2}\underbar{\ensuremath{\mu}})^{\frac{1}{2n+2}}\\
\eta_{1} & :=(\mu_{0})^{-\frac{1}{2}}-1,\\
\alpha_{1} & :=4\eta_{1}^{-1}+1,\\
\beta_{1} & :=3,\\
r_{1} & :=d_{1},
\end{aligned}
\right.\quad\left\{ \begin{aligned}\underbar{\ensuremath{\rho}} & :=\frac{1}{4}(1-\mu_{0}^{\frac{1}{2}})\sqrt{q^{-1}\underbar{\ensuremath{\mu}}},\\
\epsilon_{2} & :=\frac{1-d_{2}^{2}}{d_{2}^{2}}\theta\\
\alpha_{2} & :=\frac{1-d_{2}^{2}\underbar{\ensuremath{\rho}}^{2}}{\underbar{\ensuremath{\rho}}^{2}d_{2}^{2}},\\
\beta_{2} & :=\frac{2}{\underbar{\ensuremath{\rho}}d_{2}},\\
r_{2} & :=3/4,
\end{aligned}
\right.
\]
where $d_{2}\in(0,1)$ is a root of equation $\big(q^{2}\underbar{\ensuremath{\mu}}+d_{2}^{2n+2}-1\big)d_{2}^{-n-2}=q\underbar{\ensuremath{\mu}}$,
and keep in mind that
\[
\underbar{\ensuremath{\mu}}<q\underbar{\ensuremath{\mu}}<q^{2}\underbar{\ensuremath{\mu}}<\min\{\mu_{0},q^{-1}\underbar{\ensuremath{\mu}}\mu_{0}^{-\frac{1}{4}}\}<1.
\]
 The roles of the constants will be clear later. 

As $q^{-1}\underbar{\ensuremath{\mu}}<\underbar{\ensuremath{\mu}}$,
from the definition of $\underbar{\ensuremath{\mu}}$ there exist
$x_{0}\in\mathbb{R}_{+}^{n}$, $x\in K(x_{0},3/4)$, and
\[
\Gamma\subset Q:=Q_{\theta}(0,x_{0},1)
\]
with $q^{-1}\underbar{\ensuremath{\mu}}<|\Gamma|/|Q|<\underbar{\ensuremath{\mu}}$,
such that
\begin{equation}
\Prob^{x}\big(\hit_{\Gamma}<\ext_{Q}\big)<\eps(q\underbar{\ensuremath{\mu}})\min\big\{\eps(\mu_{0})M_{\ref{cor:hit-prob}}(c,\theta,\alpha_{1},\beta_{1},r_{1},\bd),M_{\ref{cor:hit-prob}}(c,\epsilon_{2},\alpha_{2},\beta_{2},r_{2},\bd)\big\},\label{eq:8-0-0}
\end{equation}
where $M_{\ref{cor:hit-prob}}$ is taken from Corollary \ref{cor:hit-prob}.

Applying Lemma \ref{lem:measure th} with $\mu'=q^{-1}\underbar{\ensuremath{\mu}}$,
$\mu=\mu_{0}$ and $\eta=\eta_{1}=\mu_{0}^{-1/4}-1$, and noting that
$\min\{\mu^{-1/4}\mu',\mu\}>q^{2}\underbar{\ensuremath{\mu}}$, we
have two cases: \textbf{Case I}:
\begin{equation}
|D_{2}\cap Q|\geq q^{2}\underbar{\ensuremath{\mu}}|Q|,\label{eq:8-0-1}
\end{equation}
 or \textbf{Case II}:
\[
|Q_{\theta}(\check{t},\check{z},\check{\rho})\cap\Gamma|\geq q^{2}\underbar{\ensuremath{\mu}}|Q_{\theta}(\check{t},\check{z},\check{\rho})|
\]
 for some regular hypercube $Q_{\theta}(\check{t},\check{z},\check{\rho})\subset Q$,
where $\check{\rho}\geq\underbar{\ensuremath{\rho}}=\frac{1}{4}(1-\mu_{0}^{\frac{1}{2}})\sqrt{q^{-1}\underbar{\ensuremath{\mu}}}$. 

We discuss the two cases separately.

\medskip

\textbf{Case I}. Let $\tilde{Q}:=[(1-d_{1}^{2})\theta,\theta)\times K(x_{0},d_{1})$
with $d_{1}=(1/2)\vee(1+q\underbar{\ensuremath{\mu}}-q^{2}\underbar{\ensuremath{\mu}})^{\frac{1}{2n+2}}<1$.
A simple computation yields
\begin{equation}
\big|\tilde{Q}\big|=\prod_{i=1}^{n}\frac{\big(\sqrt{x_{0}^{i}}+d_{1}\big)^{2}-\Big(\big(\sqrt{x_{0}^{i}}-d_{1}\big)\vee0\Big)^{2}}{\big(\sqrt{x_{0}^{i}}+1\big)^{2}-\Big(\big(\sqrt{x_{0}^{i}}-1\big)\vee0\Big)^{2}}\cdot d_{1}^{2}\times|Q|\geq d_{1}^{2n+2}|Q|.\label{eq:8-0-2}
\end{equation}
 Let $E:=D_{2}\cap\tilde{Q}\subset Q$. Then using (\ref{eq:8-0-1})
one has
\begin{eqnarray*}
|E| & \geq & |D_{2}\cap Q|+|\tilde{Q}|-|Q|\\
 & \geq & (q^{2}\underbar{\ensuremath{\mu}}+d_{1}^{2n+2}-1)|Q|\\
 & \ge & q\underbar{\ensuremath{\mu}}|Q|.
\end{eqnarray*}
By definition of $\eps(\cdot)$, one knows that for any $x\in K(x_{0},3/4)$,
\begin{equation}
\Prob^{x}\big[\hit_{E}<\ext_{Q}\big]\geq\eps(q\underbar{\ensuremath{\mu}}).\label{eq:8-0}
\end{equation}

Next we estimate the hitting probability when $X$ starts from the
set $E$. By the construction of $D_{2}$, one knows that, for any
$(s,y)\in E=D_{2}\cap\tilde{Q}$ and $\eta_{1}=\mu_{0}^{-1/2}-1$,
there is a regular hypercube $Q_{\theta}(t_{1},x_{1},\rho_{1})\subset Q$
such that 
\[
(s,y)\in[(t_{1}-(4\eta_{1}^{-1}+1)\theta\rho_{1}^{2},t_{1}-\theta\rho_{1}^{2})\times K(x_{1},3\rho_{1})]\cap\tilde{Q}
\]
and 
\begin{equation}
|Q_{\theta}(t_{1},x_{1},\rho_{1})\cap\Gamma|\geq\mu_{0}|Q_{\theta}(t_{1},x_{1},\rho_{1})|.\label{eq:8-1}
\end{equation}
Applying Corollary \ref{cor:hit-prob} with 
\[
c=1,\quad\epsilon_{1}=\theta,\quad\alpha_{1}=4\eta_{1}^{-1}+1,\quad\beta_{1}=3,\quad r_{1}=d_{1},
\]
and noticing $t_{1}\in[s+\theta\rho_{1}^{2},s+(4\eta^{-1}+1)\rho_{1}^{2}]$,
one obtains that
\begin{equation}
\Prob^{s,y}\big[X_{t_{1}}\in K(x_{1},3\rho_{1}/4),t_{1}\le\ext_{Q}\big]\geq M_{\ref{cor:hit-prob}}(c,\theta,\alpha_{1},\beta_{1},r_{1},\bd).\label{eq:8-2}
\end{equation}
Moreover, from (\ref{eq:8-1}) and the definition of $\eps(\cdot)$,
for any $x_{1}'\in K(x_{1},3\rho_{1}/4)$ we have 
\begin{equation}
\Prob^{t_{1},x_{1}'}\big[\hit_{\Gamma}<\ext_{Q}\big]\geq\eps(\mu_{0}).\label{eq:8-3}
\end{equation}
Combining (\ref{eq:8-2}) and (\ref{eq:8-3}), one has
\begin{eqnarray*}
\Prob^{s,y}\big[\hit_{\Gamma}<\ext_{Q}\big] & \geq & \E^{s,y}\big[\Prob^{t_{1},X_{t_{1}}}\big(\hit_{\Gamma}<\ext_{Q}\big);\{X_{t_{1}}\in K(x_{1},3\rho_{1}/4),\ext_{Q}>t_{1}\}\big]\\
 & \geq & \eps(\mu_{0})\Prob^{s,y}\big[X_{t_{1}}\in K(x_{1},3\rho_{1}/4),\ext_{Q}>t_{1}\big]\\
 & \geq & \eps(\mu_{0})M_{\ref{cor:hit-prob}}(c,\theta,\alpha_{1},\beta_{1},r_{1},\bd).
\end{eqnarray*}

Using the above relation and (\ref{eq:8-0}), we compute that
\begin{eqnarray*}
\Prob^{x}\big[\hit_{\Gamma}<\ext_{Q}\big] & \geq & \Prob^{x}\big[\hit_{E}<\hit_{\Gamma}<\ext_{Q}\big]\\
 & \geq & \E^{x}\big[\Prob^{\hit_{E},X_{\hit_{E}}}\big(\hit_{\Gamma}<\ext_{Q}\big);\hit_{E}<\ext_{\Gamma}\big]\\
 & \geq & \eps(\mu_{0})M_{\ref{cor:hit-prob}}(c,\theta,\alpha_{1},\beta_{1},r_{1},\bd)\Prob^{0,x}\big[\hit_{E}<\ext_{Q}\big]\\
 & \geq & \eps(q\underbar{\ensuremath{\mu}})\eps(\mu_{0})M_{\ref{cor:hit-prob}}(c,\theta,\alpha_{1},\beta_{1},r_{1},\bd),
\end{eqnarray*}
which contradicts (\ref{eq:8-0-0}).

\medskip

\textbf{Case II.} This case is relatively simple. Let $\tilde{Q}:=[\check{t}+\theta(1-d_{2}^{2})\check{\rho}^{2},\check{t}+\theta\check{\rho}^{2}]\times K(\check{z},d_{2}\check{\rho})$,
where $d_{2}\in(0,1)$ is a root of equation $\big(q^{2}\underbar{\ensuremath{\mu}}+d_{2}^{2n+2}-1\big)d_{2}^{-n-2}=q\underbar{\ensuremath{\mu}}$.
It is easy to verify that $\tilde{Q}$ is regular if $Q$ is regular,
and
\[
d_{2}^{-n-2}|Q_{\theta}(\check{t},\check{z},\check{\rho})|\geq|\tilde{Q}|\geq d_{2}^{2n+2}|Q_{\theta}(\check{t},\check{z},\check{\rho})|.
\]
So we have 
\begin{eqnarray*}
|\Gamma\cap\tilde{Q}| & \geq & |\Gamma\cap Q_{\theta}(\check{t},\check{z},\check{\rho})|-|Q_{\theta}(\check{t},\check{z},\check{\rho})\backslash\tilde{Q}|\\
 & \geq & q^{2}\underbar{\ensuremath{\mu}}|Q(\check{t},\check{z},\check{\rho})|-(1-d_{2}^{2n+2})|Q_{\theta}(\check{t},\check{z},\check{\rho})|\\
 & \geq & \big(q^{2}\underbar{\ensuremath{\mu}}+d_{2}^{2n+2}-1\big)|Q_{\theta}(\check{t},\check{z},\check{\rho})|\\
 & \geq & \big(q^{2}\underbar{\ensuremath{\mu}}+d_{2}^{2n+2}-1\big)d_{2}^{-n-2}|\tilde{Q}|=q\underbar{\ensuremath{\mu}}|\tilde{Q}|.
\end{eqnarray*}
According to (\ref{eq:reg-1}), we have that, for any $z'\in K(\check{z},3d_{2}\check{\rho}/4)$,
\[
\Prob^{\check{t}+\theta(1-d_{2}^{2})\check{\rho}^{2},z'}\big[\hit_{\Gamma}<\ext_{Q}\big]\geq\Prob^{\check{t}+\theta(1-d_{2}^{2})\check{\rho}^{2},z'}\big[\hit_{\Gamma}<\ext_{\tilde{Q}}\big]\geq\eps(q\underbar{\ensuremath{\mu}}).
\]
Applying Corollary \ref{cor:hit-prob} on $[0,\alpha_{2}l_{2}^{2}]$
with
\[
c=1,\quad l_{2}=\check{\rho}d_{2},\quad\epsilon_{2}=\frac{1-d_{2}^{2}}{d_{2}^{2}}\theta,\quad\alpha_{2}=\frac{1-d_{2}^{2}\underbar{\ensuremath{\rho}}^{2}}{\underbar{\ensuremath{\rho}}^{2}d_{2}^{2}}\theta,\quad\beta_{2}=\frac{2}{\underbar{\ensuremath{\rho}}d_{2}},\quad r_{2}=\frac{3}{4},
\]
and noticing $\check{t}+\theta(1-d_{2}^{2})\check{\rho}^{2}\in[\epsilon_{2}l_{2}^{2},\alpha_{2}l_{2}^{2}]$,
we have that, for any $x\in K(x_{0},3/4)$, 
\begin{align*}
 & \Prob^{x}\big[\hit_{\Gamma}<\ext_{Q}\big]\\
\geq\, & \E^{x}\bigg[\begin{array}{c}
\Prob^{\check{t}+\theta(1-d_{2}^{2})\check{\rho}^{2},X_{\check{t}+\theta(1-d_{2}^{2})\check{\rho}^{2}}}\big[\hit_{\Gamma}<\ext_{Q}\big];\\
X_{\check{t}+\theta(1-d_{2}^{2})\check{\rho}^{2}}\in K(\check{z},3d_{2}\check{\rho}/4),\ext_{Q}>\check{t}+\theta\check{\rho}^{2}(1-d_{2}^{2})
\end{array}\bigg]\\
\geq\, & \eps(q\underbar{\ensuremath{\mu}})\Prob^{x}\Big[X_{\check{t}+\theta(1-d_{2}^{2})\check{\rho}^{2}}\in K(\check{z},3d_{2}\check{\rho}/4),\ext_{Q}>\check{t}+\theta(1-d_{2}^{2})\check{\rho}^{2}\Big]\\
\geq\, & \eps(q\underbar{\ensuremath{\mu}})M_{\ref{cor:hit-prob}}(c,\epsilon_{2},\alpha_{2},\beta_{2},r_{2},\bd),
\end{align*}
which also contradicts (\ref{eq:8-0-0}). Therefore, Theorem \ref{thm:main2}
is proved if $Q$ is regular.

\subsection{When $Q$ is not regular}

The idea is to shift and shrink $Q=Q_{\theta}(0,x_{0},1)$ properly
so the new \emph{$Q_{\theta}(0,\widehat{x}_{0},2/3)$ is regular}
and
\[
K(x_{0},1/6)\subset K(\widehat{x}_{0},1/2),\quad Q_{\theta}(0,\widehat{x}_{0},2/3)\subset Q_{\theta}(0,x_{0},1).
\]
This can be easily realized by the following choice of $\widehat{x}_{0}$:
for each $i=1,\dots,n$,
\[
\widehat{x}_{0}^{i}=\begin{cases}
0 & \quad\text{if}\ \sqrt{x_{0}^{i}}\in[0,1/3),\\{}
[\sqrt{x_{0}^{i}}+1/3]^{2} & \quad\text{if}\ \sqrt{x_{0}^{i}}\in[1/3,1),\\
x_{0}^{i} & \quad\text{if}\ \sqrt{x_{0}^{i}}\in[1,\infty).
\end{cases}
\]
Applying the previous result to $Q_{\theta}(0,\widehat{x}_{0},2/3)$
we conclude Theorem \ref{thm:main2}. 

\section{Proof of Theorem \ref{thm:inv-meas}}

Let us first prove that the transition semigroup $P=(P_{t})_{t\ge0}$
associated with the Markov process $X$ is strongly Feller under Condition
(\textbf{C}). For any $\vf\in\mathcal{B}_{b}(\R_{+}^{n})$ and $(t,x)\in(0,1)\times\R_{+}^{n}$,
let
\begin{equation}
u(t,x):=P_{1-t}\vf(x)=\E^{x}\big[\vf(X_{1-t})\big].\label{eq:SF-01}
\end{equation}
 In view of the Markov property of $X$, one has that for any $s\in(0,t)$
and $x\in\R_{+}^{n}$,
\begin{align*}
\E^{x}\big[u(t,X_{t})\big|\mathcal{F}_{s}\big] & =\E^{x}\big[\E^{X_{t}}\big[\vf(X_{1-t})\big]\big|\mathcal{F}_{s}\big]\\
 & =\E^{x}\big[\E^{x}\big[\vf(X_{1})\big|\mathcal{F}_{t}\big]\big|\mathcal{F}_{s}\big]\\
 & =\E^{x}\big[\vf(X_{1})\big|\mathcal{F}_{s}\big]\\
 & =\E^{X_{s}}\big[\vf(X_{1-s})\big]=u(s,X_{s}),\quad\Prob^{x}\text{-a.s.}
\end{align*}
This means that $u$ satisfies Condition (\textbf{U}), and from Theorem
\ref{thm:main1}, $u(t,\cdot)$ is H\"older continuous for any $t\in(0,1)$
and so is $P_{1-t}\vf(\cdot)$. This yields the strong Feller property
of $P$. 

The main tools for existence and uniqueness of invariant probability
measures of $P$ are the Krylov-Bogoliubov existence theorem and the
Khas'minskii-Doob theorem (see \citep[Sections 4.1 and 4.2]{da1996ergodicity}).
For uniqueness we need another concept: the semigroup $P$ is said
to be \textit{irreducible} at time $t>0$ if, for arbitrary nonempty
open set $\Gamma$ and all $x\in\R_{+}^{n}$,
\[
P_{t}\bm{1}_{\Gamma}(x)=\mathbb{P}^{x}[X_{t}\in\Gamma]>0.
\]
Evidently, the irreducibility of $P$ follows from Lemma \ref{prop:smallcube}.
For existence we need the following tightness result for the law of
$X$.
\begin{lem}
\label{Lem:tight}Under the assumption of Theorem \ref{thm:inv-meas},
for each $x\in\R_{+}^{n}$, $\eps>0$ there exists a constant $N=N(\eps,\lambda,x)>0$
such that
\[
\Prob^{x}\{|X_{t}|>N\}<\eps\quad\forall\,t>0.
\]
 
\end{lem}

\begin{proof}
It suffices to prove that for each $i=1,\dots,n$ we have $\Prob^{x}[X_{t}^{i}>N]<\eps$.
Define
\[
T_{k}=\inf\{t>0:|X_{t}|=k\},\quad k=1,2,\dots
\]
Then by the Fubini theorem and using (\ref{eq:inv-meas-1}) we have
\begin{align*}
\E^{x}[X_{t\wedge T_{k}}^{i}] & =x^{i}+\E^{x}\int_{0}^{t\wedge T_{k}}b^{i}(X_{s})\vd s\\
 & \le|x|+\E^{x}\int_{0}^{t}(-\lambda X_{s\wedge T_{k}}^{i})\vd s\\
 & =|x|-\lambda\int_{0}^{t}\E^{x}[X_{s\wedge T_{k}}^{i}]\vd s
\end{align*}
which along with the Gr\"onwall inequality implies
\[
\E^{x}[X_{t\wedge T_{k}}^{i}]\le|x|\me^{-\lambda t}<|x|,\quad\forall t>0.
\]
Since $X(\omega)\in C([0,\infty);\R_{+}^{n})$, $T_{k}(\omega)\uparrow\infty$
as $k\uparrow\infty$ for each $\omega\in\PS$, then by Fatou's lemma
we have $\E^{x}[X_{t}^{i}]<|x|+\lambda^{-2}$, thus 
\[
\Prob^{x}[X_{t}^{i}>N]<\big(|x|+\lambda^{-2}\big)N^{-1}
\]
from Chebyshev's inequality. The proof is then easily concluded. 
\end{proof}
Now let us complete the proof of Theorem \ref{thm:inv-meas}. According
to the Krylov-Bogoliubov theorem, existence of invariant measures
of $P$ follows from its (strong) Feller property and tightness (due
to Lemma \ref{Lem:tight}). Moreover, $P$ is irreducible due to Lemma
\ref{prop:smallcube}, which combining with the strong Feller property
yields the uniqueness (and also ergodicity) of the invariant measure
by means of the Khas'minskii-Doob theorem. The proof is complete.


\bibliographystyle{amsalpha}
\bibliography{holder}

\end{document}